\documentclass[11pt]{article}

\usepackage{amsmath}
\usepackage{amssymb}
\usepackage{amsfonts}
\usepackage{amsthm}
\usepackage{array}
\usepackage{bm}
\usepackage[shortlabels]{enumitem}
\usepackage{stmaryrd}

\usepackage{amsxtra}
\usepackage{array}
\usepackage{mathrsfs}
\usepackage{slashed}
\usepackage{xcolor}
\usepackage{tikz-cd}
\usepackage{tkz-euclide} 
\usepackage{pgfplots}
\pgfplotsset{compat=newest}
\usepackage{caption}
\usepackage{mathtools}

\newcolumntype{C}[1]{>{\centering\hspace{0pt}}p{#1}}
\usepackage[margin=1in]{geometry}

\newcommand\restr[2]{{
  \left.\kern-\nulldelimiterspace 
  #1 
  \vphantom{\big|} 
  \right|_{#2} 
  }}

\newcommand{\GL}{\mathrm{GL}}

\newcommand{\SO}{\mathrm{SO}}
\newcommand{\Spin}{\mathrm{Spin}}
\newcommand{\U}{\mathrm{U}}
\newcommand{\SU}{\mathrm{SU}}
\newcommand{\Sp}{\mathrm{Sp}}

\newcommand{\Ad}{\mathrm{Ad}}

\newcommand{\diag}{\mathrm{diag}}

\newcommand{\Spinor}{\slashed{S}}

\newcommand{\Sym}{\mathrm{Sym}}

\newcommand{\Z}{\mathbb{Z}}

\newcommand{\R}{\mathbb{R}}
\newcommand{\C}{\mathbb{C}}

\newcommand{\CP}{\mathbb{CP}}

\newcommand{\vol}{\mathrm{vol}}

\newcommand{\LB}{[\![}
\newcommand{\RB}{]\!]}

\newtheorem{thm}{Theorem}[section]
\newtheorem{prop}[thm]{Proposition}
\newtheorem{lem}[thm]{Lemma}
\newtheorem{cor}[thm]{Corollary}

\theoremstyle{definition}
\newtheorem{defn}[thm]{Definition}
\newtheorem{example}[thm]{Example}
\newtheorem{rmk}[thm]{Remark}

\newtheorem{setup}[thm]{Setup}

\usepackage[nottoc]{tocbibind}
\numberwithin{equation}{section}

\usepackage{hyperref}
\hypersetup{colorlinks, linkcolor=blue, citecolor=blue}

\DeclarePairedDelimiter\abs{\lvert}{\rvert}%
\title{Gauge theory on $T^*\CP^2$: explicit $\Sp(2)$-instantons, \\
HYM connections,
and $\Spin(7)$-instantons}
\author{Izar Alonso, Jesse Madnick, Emily Autumn Windes}
\date{September 2026}

\newcommand{\Addresses}
{{  \bigskip
\noindent	(Izar Alonso) \textsc{ICMAT-CSIC} \par\nopagebreak
\noindent	\textsc{Madrid, Spain}\par\nopagebreak
\noindent	\texttt{izar.alonso@icmat.es} \\

\noindent	(Jesse Madnick) \textsc{Seton Hall University} \par\nopagebreak
\noindent	\textsc{South Orange, NJ, United States}\par\nopagebreak
\noindent	\texttt{jesse.ochs.madnick@gmail.com} \\

\noindent	(Emily Autumn Windes) \par\nopagebreak
\noindent	\textsc{Princeton, NJ, United States} \par\nopagebreak
\noindent	\texttt{emily.a.windes@gmail.com} \\
		
}}

\begin{document}

\maketitle

\begin{abstract}
We construct and classify $\SU(3)$-invariant 
$\Sp(2)$-instantons with gauge groups $S = S^1$ and $S = \SO(3)$ over the Calabi manifold $X = T^*\CP^2$, the unique non-flat, complete, irreducible hyperk\"ahler $8$-manifold admitting a cohomogeneity-one action by a compact simple Lie group.
Moreover, in the case of $S = S^1$, we also classify $\SU(3)$-invariant primitive Hermitian Yang-Mills connections, and the $\SU(3)$-invariant $\Spin(7)$-instantons over $X$ in the following sense. Letting $\Phi_I$, $\Phi_J$, $\Phi_K$ denote the $\Spin(7)$-structures on $X$ induced from the complex structures $I,J,K$ in the hyperk\"{a}hler triple, we prove that on each invariant $S^1$-bundle $\widetilde{E}_{k} \to X$, $k \in \mathbb{Z}$, the space of invariant $\Spin(7)$-instantons with respect to $\Phi_L$ forms a one-parameter family modulo gauge.  Moreover, every pair of one-parameter families of $\Phi_I$-, $\Phi_J$-, and $\Phi_K$-$\Spin(7)$-instantons intersects only at the unique invariant $\Sp(2)$-instanton on $\widetilde{E}_k$, which is non-flat when $k \neq 0$. 
In the case of $S= \SO(3)$ and up to gauge transformations, we find exactly one invariant $\Sp(2)$-instanton on each bundle $\widetilde{E}_k$, $k \in \mathbb{Z}_{\geq 0}$, which is reducible, and a unique invariant $\Sp(2)$-instanton on the bundle $\widetilde{E}_\wedge$, which is irreducible. We also classify $\SU(3)$-invariant primitive Hermitian Yang-Mills connections on $\widetilde{E}_0$, which modulo gauge reduce to the flat connection.
\end{abstract}

 \tableofcontents

\section{Introduction}

\subsection{Background}

\indent \indent The geometric structure of hyperk\"{a}hler $8$-manifolds $(X, g, (I,J,K))$ is remarkably rich.  Essentially by definition, we may view $X$ as a Riemannian manifold via the metric $g$, as a complex manifold with respect to any of the complex structures $I,J,K$, or a symplectic manifold via the K\"{a}hler forms $\omega_I$, $\omega_J$, $\omega_K$.  Moreover, for any of the complex structures $L \in \{I,J,K\}$, we may also define a non-vanishing $(4,0)$-form $\Upsilon_L$, thereby endowing $X$ with three distinct Calabi--Yau structures.  In turn, these Calabi--Yau structures  induce $\Spin(7)$-structures $\Phi_I, \Phi_J, \Phi_K \in \Omega^4(X)$ via the formula $\Phi_L = \frac{1}{2}\omega_L^2 + \mathrm{Re}(\Upsilon_L)$. \\
\indent This means that hyperk\"{a}hler $8$-manifolds are host to a wealth of natural gauge-theoretic PDE.  Indeed, for connections $A$ on bundles $E \to X$, one may ask for $A$ to be, say, a $\Spin(7)$-instanton with respect to a particular $\Phi_L$, or a primitive Hermitian Yang--Mills (pHYM) connection with respect to a given $L$.  Even more stringently, one may ask for $A$ to be an \emph{$\Sp(2)$-instanton} (or \emph{hyperholomorphic connection}), which means that $A$ is pHYM with respect to $I$, $J$, and $K$ simultaneously:
$$A \text{ is an }\Sp(2)\text{-instanton } \ \iff \ A \text{ is pHYM with respect to }I,J,K.$$
In view of the inclusions $\Sp(2) \leq \SU(4) \leq \Spin(7)$, these gauge-theoretic conditions are nested accordingly:
\begin{equation} \label{eq:Implication}
\Sp(2)\text{-instanton} \ \implies \ \text{pHYM connection for }L \ \implies \ \Spin(7)\text{-instanton for }\Phi_L.
\end{equation}
\indent While there is a vast literature on pHYM connections, and a growing body of work on $\Spin(7)$-instantons in recent years, there are relatively few studies of $\Sp(2)$-instantons at present.  In particular, there seem to be few explicit $\Sp(2)$-instantons in the literature. As such, one of the central aims of this work is to construct concrete examples of such objects. \\
\indent Relatedly, it is interesting to ask under what circumstances either of the implications in (\ref{eq:Implication}) might be reversed.  When $X$ is compact, the answer is quite simple.  Indeed, provided that $X$ admits at least one pHYM connection with respect to $L$ on a fixed bundle, Lewis \cite{lewis1999spin} proved that the $\Phi_L$-$\Spin(7)$-instantons coincide with the $L$-pHYM connections on that bundle.  Similarly, if $X$ admits at least one $\Sp(2)$-instanton on a fixed bundle, then a result of Verbitsky \cite{verbitsky1993hyperholomorphic} implies that the class of $L$-pHYM connections (for \emph{any} $L$) coincides with that of the $\Sp(2)$-instantons on said bundle. \\
\indent When $X$ is asymptotically conical (AC) of rate $\nu \leq -\frac{10}{3}$, there exist analogous results.  That is, assuming that an AC $L$-pHYM connection exists, Papoulias \cite{papoulias2022spin} showed that the $\Phi_L$-$\Spin(7)$-instantons with sufficiently fast decay coincide with the AC $L$-pHYM connections.  Moreover, the rate restriction $\nu \leq -\frac{10}{3}$ is sharp.  Similarly, the last two authors \cite{madnick2025sp} showed that if an AC $\Sp(2)$-instanton exists, then the AC $L$-pHYM connections with sufficiently fast decay coincide with the AC $\Sp(2)$-instantons.  It is not known whether the rate restriction $\nu \leq -\frac{10}{3}$ in \cite{madnick2025sp} is sharp or could be improved.

\subsection{Main results}

\indent \indent In this paper, we will construct explicit examples of $\Sp(2)$-instantons, pHYM connections, and $\Spin(7)$-instantons on the 8-manifold $X = T^*\CP^2$ equipped with its Calabi hyperk\"{a}hler structure $(g, (I,J,K))$.  There are two reasons to prioritize this particular space.  First, $X$ has a high degree of symmetry, making it a natural testing ground for examples and conjectures.  Indeed, $X$ is cohomogeneity-one under an $\SU(3)$-action, with the principal orbits equivariantly diffeomorphic to the Aloff--Wallach manifold $N_{1,-2} = \SU(3)/\U(1)$, and with the zero section $\CP^2 = \SU(3)/\U(2)$ serving as the unique singular orbit.  In fact, a result of Dancer--Swann \cite{dancer1997hyperkahler} asserts that $T^*\CP^2$ is the unique non-flat, complete, irreducible hyperk\"ahler $8$-manifold admitting a cohomogeneity-one action by a compact simple Lie group.  Second, since $X$ is asymptotically conical of rate $\nu = -2 > -\frac{10}{3}$, it is outside the scope of the Papoulias \cite{papoulias2022spin} and Madnick--Windes \cite{madnick2025sp} results. \\
\indent In $\S$\ref{sec:InvarInstS1}, for each $L \in \{I,J,K\}$, we completely classify $\SU(3)$-invariant $\Phi_L$-$\Spin(7)$-instantons over $X$ having gauge group $S = S^1$.  As we now explain, this is essentially a three-step process.  First, we classify the homogeneous $S^1$-bundles $P_n \to N_{1,-2}$, finding one for each integer $n \in \Z$, pull those back to $S^1$-bundles called $E_n \to X^*$ over the principal locus $X^* = T^*\CP^2 - \CP^2$, and then prove that $E_n$ extends equivariantly across the zero section to an $S^1$-bundle $\widetilde{E}_k \to X$ if and only if $n = -k$.  Moreover, such an extension is unique. \\
\indent Second, we classify the $\SU(3)$-invariant $\Phi_L$-$\Spin(7)$-instantons on $E_n \to X^*$ by solving the appropriate ODE system.  For each $L \in \{I, J, K\}$, we find an explicit $3$-parameter family of such connections.  Finally, 
using the Tao--Tian removable singularity theorem \cite[Theorem 1.1]{tao2004singularity}, we verify that the class of invariant $\Phi_L$-$\Spin(7)$-instantons that extends across the zero section forms a one-parameter family.  In fact, more is true:

\begin{thm} \label{thm:Main1} On each $S^1$-bundle $\widetilde{E}_k \to T^*\CP^2$ for $k \in \Z$, and for each $L \in \{I,J,K\}$, the collection of $\SU(3)$-invariant $\Phi_L$-$\Spin(7)$-instantons on $\widetilde{E}_k$ is a one-parameter family $\{A_{L, t}\}_{t \in \R}$ modulo gauge transformations. \\
\indent Moreover, among the three families $\{A_{I,t}\}$, $\{A_{J,t}\}$, and $\{A_{K,t}\}$, each pair intersects only at $A_{I,0} = A_{J,0} = A_{K,0}$.  In fact, $A_{L,0}$ is the unique connection in the $A_{L,t}$ family that is primitive Hermitian Yang--Mills with respect to $L$.
\end{thm}

\begin{cor} On each $S^1$-bundle $\widetilde{E}_k \to T^*\CP^2$ for $k \in \Z$, there exists a unique $\SU(3)$-invariant $\Sp(2)$-instanton (namely, $A_{I,0} = A_{J,0} = A_{K,0}$).  For $k \neq 0$, this $\Sp(2)$-instanton is not flat.
\end{cor}

\indent In $\S$\ref{sec:InvarInstSO3}, we study $\SU(3)$-invariant instantons over $X$ with gauge group $S = \SO(3)$.  As in the $S^1$ case, we classify the homogeneous $\SO(3)$-bundles over $N_{1,-2}$, finding one for each $n \in \Z_{\geq 0}$, and thereby obtain $\SO(3)$-bundles $E_n \to X^*$ over the principal locus.  We show that all $E_n$ extend equivariantly across the zero section. 
\\
\indent We then consider separately the cases of $n \neq 0, 3$, $n=0$ and $n = 3$.  For $n \neq 0, 3$, the $\SU(3)$-invariant connections on $E_n \to X^*$ essentially reduce to those on $S^1$-bundles.  Accordingly, the analogue of Theorem \ref{thm:Main1} holds in this setting.  For $n = 0$, the $\Phi_L$-$\Spin(7)$-instanton equations are considerably more formidable.  As such, we restrict our attention to the subclass of $L$-pHYM connections on $E_0 \to X^*$, obtaining a $3$-parameter family.  
However, in that family, the only element that extends across the zero section is the canonical connection, which is flat.  
Finally, for $n=3$,  we restrict our attention to $\Sp(2)$-instantons. We show that the bundle $E_3 \to X^*$ admits  two distinct extensions, which we label $\widetilde{E}_3$ and $\widetilde{E}_\wedge$. Modulo gauge transformations, we find exactly one $\Sp(2)$-instanton on each of them.
To summarize:

\begin{thm} ${}$
\begin{enumerate}[(a)]
\item On each $\SO(3)$-bundle $\widetilde{E}_k \to T^*\CP^2$ for $k \neq 0, 3$, and for each $L \in \{I,J,K\}$, the collection of $\SU(3)$-invariant $\Phi_L$-$\Spin(7)$-instantons on $\widetilde{E}_k$ is a one-parameter family $\{A_{L, t}\}_{t \in \R}$ modulo gauge transformations.  Moreover, $A_{I,0} = A_{J,0} = A_{K,0}$ is the unique connection in the $A_{L,t}$ family that is primitive Hermitian Yang--Mills with respect to $L$, and it is an $\Sp(2)$-instanton.
\item On the $\SO(3)$-bundle $\widetilde{E}_0 \to T^*\CP^2$, for each $L \in \{I,J,K\}$, the only $\SU(3)$-invariant connection that is primitive Hermitian Yang--Mills with respect to $L$ is the (flat) canonical connection.
\item  On the $\SO(3)$-bundle $\widetilde{E}_3 \to T^*\CP^2$, there is a unique (reducible) $\SU(3)$-invariant $\Sp(2)$-instanton; on the $\SO(3)$-bundle $\widetilde{E}_\wedge \to T^*\CP^2$, there is a unique $\SU(3)$-invariant $\Sp(2)$-instanton up to gauge transformation.
\end{enumerate}
\end{thm}

\indent This work is organized as follows.  In $\S$\ref{sec:InstHighDim} we introduce $\Sp(2)$-, $\SU(4)$-, and $\Spin(7)$-instantons, and study the relationships between those connections on hyperk\"ahler 8-manifolds. Note that we use the term ``$\SU(4)$-instanton" for ``primitive HYM connection," and abbreviate $\Phi_L$-$\Spin(7)$ as simply $L$-$\Spin(7)$.  In $\S$\ref{sec:CalabiGeometry}, we discuss the geometry of the Calabi manifold $T^*\CP^2$.  Section \ref{sec:InvarBundConn} consists of additional preliminaries on homogeneous bundles and invariant connections.  Finally, in $\S$\ref{sec:InvarInstS1} we classify instantons over $T^*\CP^2$ with gauge group $S = S^1$, and in $\S$\ref{sec:InvarInstSO3} we construct instantons with gauge group $S = \SO(3)$. \\

\noindent \textbf{Acknowledgements:} We gratefully acknowledge helpful conversations with Gavin Ball, Benoit Charbonneau, Andrew Dancer, Jason Lotay, Gon\c{c}alo Oliveira, Enric Sol\'e-Farr\'e, Jakob Stein, Henrique S\'{a} Earp, and McKenzie Wang.  We also thank the American Institute of Mathematics and the Okinawa Institute of Science and Technology.

\section{Instantons in higher dimensions} \label{sec:InstHighDim}

\indent \indent The purpose of this section is to set foundations.  In $\S$\ref{sub:Prelim}, we recall the notions of ``$H$-instanton" and ``$\Omega$-ASD instanton" on suitable Riemannian manifolds of dimension $n \geq 4$.  In particular, we highlight three fundamental classes: the $\SU(m)$-instantons, $\Sp(m)$-instantons, and $\Spin(7)$-instantons.  Then, in $\S$\ref{sub:GaugeHK}, we specialize to the setting of hyperk\"{a}hler 8-manifolds, establishing relationships between $\Sp(2)$-, $\SU(4)$-, and $\Spin(7)$-instantons.

\subsection{Preliminaries} \label{sub:Prelim}

\indent \indent We begin by recalling two standard notions of ``instanton" over manifolds of dimension at least four.  For this, let $(X^n, g)$ be an oriented Riemannian $n$-manifold with $n \geq 4$, let $E \to X$ be a Hermitian vector bundle, and let $\mathrm{ad}_E \to X$ denote the adjoint bundle.  Let $A$ be a connection on $E$, and let $F_A \in \Omega^2(X; \mathrm{ad}_E) = \Gamma(\Lambda^2(T^*X) \otimes \mathrm{ad}_E)$ denote the curvature of $A$.

\begin{defn} Suppose $X$ is equipped with an $(n-4)$-form $\Omega \in \Omega^{n-4}(X)$.  We say that $A$ is an \emph{$\Omega$-anti-self-dual} (or $\Omega$-ASD) \emph{instanton} if its curvature satisfies
\begin{equation*}
\ast F_A = -\Omega \wedge F_A.
\end{equation*}
\end{defn}

We note that $\Omega$-ASD instantons are \emph{$\Omega$-Yang--Mills}, i.e.\ they satisfy the \emph{$\Omega$-Yang--Mills equation}:
$$
d_A (*F_A) = -d\Omega \wedge F_A.
$$
For all the forms $\Omega$ considered in this paper, $\Omega$-ASD instantons are also absolute minimisers of the following functional on compact manifolds
$$
\text{YM}_\Omega (A) = \int_X |F_A |^2 \text{vol} - \int_X \text{tr}(F_A \wedge F_A ) \wedge \Omega.
$$
Moreover, the forms $\Omega$ that we consider are all closed. In this case, $\Omega$-ASD instantons are also \emph{Yang--Mills connections}, meaning that $d_A(* F_A) = 0$.

\begin{defn} Suppose $X$ is equipped with an $H$-structure, where $H \leq \SO(n)$ is a closed subgroup and $\mathfrak{h} \subset \mathfrak{so}(n)$ denotes its Lie algebra.  (For example, if the holonomy group of $g$ is contained in $H$, then $X$ admits an $H$-structure.)  Splitting $\Lambda^2(\R^n)^* \cong \mathfrak{so}(n) = \mathfrak{h} \oplus \mathfrak{h}^\perp$ with respect to the Killing form yields a decomposition
$$\Lambda^2(T^*X) \cong \Lambda^2_{\mathfrak{h}} \oplus \Lambda^2_{\mathfrak{h}^\perp}.$$
We say $A$ is an \emph{$H$-instanton} if its curvature takes values in $\mathfrak{h}$ in the sense that
\begin{equation*}
F_A \in \Gamma(\Lambda^2_{\mathfrak{h}} \otimes \mathrm{ad}_E).
\end{equation*}
\end{defn}

 We now provide three examples of the above definitions.  All three will play an important role in this work.
 
\begin{example}[$H = \SU(m)$] Let $(X^{2m}, g, J, \omega_J, \Upsilon_J)$ be a Calabi--Yau $2m$-manifold, where $g$ is the Riemannian metric, $J$ is a $g$-orthogonal parallel complex structure, $\omega_J \in \Omega^2(X)$ is the K\"{a}hler form, and $\Upsilon_J \in \Omega^{m,0}(X)$ is the holomorphic volume form.  It is well-known that the following two conditions on a connection $A$ are equivalent:
\begin{enumerate}[(a)]
\item $A$ is an $\SU(m)$-instanton: $F_A \in \Gamma(\Lambda^2_{\mathfrak{su}(m)} \otimes \mathrm{ad}_E)$.
\item $A$ is $\Omega_J$-ASD for $\Omega_J = \frac{1}{(m-2)!}\omega_J^{m-2} \in \Omega^{2m-4}(X)$.
\end{enumerate}
We often refer to such connections as \emph{$J$-$\SU(m)$-instantons}, especially when $X$ is equipped with several different Calabi--Yau structures (e.g., see the subsequent example).  In the literature, the $\C$-linear extensions of such connections are called \emph{primitive Hermitian Yang--Mills (pHYM)}.
\end{example}

\begin{example}[$H = \Sp(m)$] \label{example:Sp} Let $(X^{4m}, g, (I,J,K), (\omega_I, \omega_J, \omega_K))$ be a hyperk\"{a}hler $4m$-manifold.  Here, $g$ is the Riemannian metric, $(I,J,K)$ is a triple of $g$-orthogonal parallel complex structures satisfying $IJ = K$, and $\omega_L \in \Omega^2(X)$ denotes the K\"{a}hler form of $L \in \{I,J,K\}$, meaning that $\omega_L(X,Y) = g(LX,Y)$.  The following conditions are equivalent \cite{madnick2025sp}:
\begin{enumerate}[(a)]
\item $A$ is an $\Sp(m)$-instanton: $F_A \in \Gamma(\Lambda^2_{\mathfrak{sp}(m)} \otimes \mathrm{ad}_E)$.
\item $A$ is $\Theta$-ASD for $\Theta = \frac{1}{(2m-1)!}(\omega_I^2 + \omega_J^2 + \omega_K^2)^{m-1} \in \Omega^{4m-4}(X)$.
\item $A$ is simultaneously $\Omega_I$-, $\Omega_J$-, and $\Omega_K$-ASD, where $\Omega_L := \frac{1}{(2m-2)!}\omega_L^{2m-2} \in \Omega^{4m-4}(X)$.
\item $A$ is simultaneously an $I$-$\SU(2m)$-, $J$-$\SU(2m)$-, and $K$-$\SU(2m)$-instanton.
\end{enumerate}
We remark that in dimension four (i.e., when $m = 1$), the $\Sp(1)$-instanton condition coincides with the classical ASD equation.
\end{example}

 \begin{example}[$H = \Spin(7)$] Let $(X^8, g, \Phi)$ be a $\Spin(7)$-manifold, where $g$ is the Riemannian metric and $\Phi \in \Omega^4(X)$ is the Cayley $4$-form.  It is well-known that the following two conditions on a connection $A$ are equivalent:
\begin{enumerate}[(a)]
\item $A$ is a $\Spin(7)$-instanton: $F_A \in \Gamma(\Lambda^2_{\mathfrak{spin}(7)} \otimes \mathrm{ad}_E)$.
\item $A$ is $\Phi$-ASD.
\end{enumerate}
\end{example}

\subsection{Gauge theory on hyperk\"{a}hler $8$-manifolds} \label{sub:GaugeHK}

\indent \indent This work concerns gauge theory on hyperk\"{a}hler $8$-manifolds.  As we now explain, there are various classes of ``$H$-instantons" that one may consider in such a setting, essentially due to the inclusions $\Sp(2) \leq \SU(4) \leq \Spin(7)$.

\indent Let $(X^8, g, (I,J,K), (\omega_I, \omega_J, \omega_K))$ be a hyperk\"{a}hler $8$-manifold, let $E \to X$ be a Hermitian vector bundle, and let $A$ be a connection on $E$.  By the remarks in Example \ref{example:Sp}, we see that
\begin{equation} \label{eq:Sp2SU4}
A \text{ is an }\Sp(2)\text{-instanton} \ \iff \ A \text{ is an }I\text{-}\SU(4)\text{-}, \,J\text{-}\SU(4)\text{-},\text{ and } K\text{-}\SU(4)\text{-instanton.}
\end{equation}
Now, for each $L \in \{I,J,K\}$, we can view $X$ as a Calabi--Yau $4$-fold via the Calabi--Yau structure $(g, L, \omega_L, \Upsilon_L)$, where the holomorphic volume form $\Upsilon_I \in \Omega^{4,0}(X)$ is given by $\Upsilon_I = \frac{1}{2}(\omega_J + i\omega_K)^2$ (and similarly for cyclic permutations of $\{I,J,K\}$).  In turn, each Calabi--Yau structure induces a $\Spin(7)$-structure $\Phi_L \in \Omega^4(X)$ via the formula
\begin{equation} \label{eq:Spin7Phase1}
\Phi_L = \frac{1}{2}\omega_L^2 + \mathrm{Re}(\Upsilon_L),
\end{equation}
so
\begin{align} \label{eq:CalabiSpin7}
\Phi_I & = \frac{1}{2}(\omega_I^2 + \omega_J^2 - \omega_K^2), & \Phi_J & = \frac{1}{2}(-\omega_I^2 + \omega_J^2 + \omega_K^2), & \Phi_K & = \frac{1}{2}(\omega_I^2 - \omega_J^2 + \omega_K^2).
\end{align}
We shall say that $A$ is an \emph{$L$-$\Spin(7)$-instanton} if $A$ is a $\Spin(7)$-instanton with respect to $\Phi_L$.  It is well-known that (thanks to the inclusion $\mathfrak{su}(4) \subset \mathfrak{spin}(7)$), one has the implication
\begin{equation} \label{eq:SU4Spin7}
A \text{ is an }L\text{-}\SU(4)\text{-instanton } \ \implies A \text{ is an }L\text{-}\Spin(7)\text{-instanton.}
\end{equation}
In summary, every $\Sp(2)$-instanton is an $L$-$\SU(4)$-instanton, and every $L$-$\SU(4)$-instanton is an $L$-$\Spin(7)$-instanton. \\
\indent In the other direction, we may ask what conditions result from requiring that a connection $A$ be, say, simultaneously an $I$-$\SU(4)$- and $J$-$\SU(4)$-instanton, or simultaneously an $I$-$\Spin(7)$- and $J$-$\Spin(7)$-instanton.  The following two results address all such pairwise and triple intersections.

\begin{prop} \label{prop:SU4Spin7A} Let $X$ be a hyperk\"{a}hler $8$-manifold, and let $A$ be a connection on a Hermitian vector bundle.   The following are equivalent:
\begin{enumerate}[(a)]
\item $A$ is an $I$-$\SU(4)$-instanton.
\item $A$ is both an $I$-$\Spin(7)$-instanton and a $K$-$\Spin(7)$-instanton.
\end{enumerate}
The analogous result holds for cyclic permutations of $\{I,J,K\}$.
\end{prop}

\begin{proof} Let $A$ be a connection on a Hermitian vector bundle.  Observing that $\Phi_K + \Phi_I = \omega_I^2$, we have
\begin{equation} \label{eq:PhiKI}
\textstyle F_A \wedge \Phi_K + F_A \wedge \Phi_I = 2(F_A \wedge \frac{1}{2}\omega_I^2).
\end{equation}
\indent (a) $\implies$ (b).  Suppose $A$ is an $I$-$\SU(4)$-instanton.  Then $2(F_A \wedge \frac{1}{2}\omega_I^2) = -2\ast\! F_A$.  Moreover, by  (\ref{eq:SU4Spin7}), $A$ is an $I$-$\Spin(7)$-instanton, so $F_A \wedge \Phi_I = -\ast\! F_A$.  Equation (\ref{eq:PhiKI}) now implies that $A$ is a $K$-$\Spin(7)$-instanton. \\
\indent (b) $\implies$ (a). Suppose $A$ is both an $I$-$\Spin(7)$-instanton and a $K$-$\Spin(7)$-instanton.  Then (\ref{eq:PhiKI}) implies $-2\ast \!F_A = 2(F_A \wedge \frac{1}{2}\omega_I^2)$, from which it follows that $A$ is an $I$-$\SU(4)$-instanton.
\end{proof}

\begin{cor} \label{cor:Sp2Spin7relation} Let $X$ be a hyperk\"{a}hler $8$-manifold, and let $A$ be a connection on a Hermitian vector bundle.   The following are equivalent:
\begin{enumerate}[(a)]
\item $A$ is an $\Sp(2)$-instanton.
\item $A$ is simultaneously a $J$-$\SU(4)$- and $K$-$\SU(4)$-instanton.
\item $A$ is simultaneously an $I$-$\Spin(7)$, $J$-$\Spin(7)$-, and $K$-$\Spin(7)$-instanton.
\end{enumerate}
\end{cor}

\begin{proof} The implication (a) $\implies$ (b) is immediate.  The implication (b) $\implies$ (c) follows from Proposition \ref{prop:SU4Spin7A}.  Finally, (c) $\implies$ (a) also follows from Proposition \ref{prop:SU4Spin7A} and (\ref{eq:Sp2SU4}).
\end{proof}

\begin{rmk}[Phases] Formula (\ref{eq:Spin7Phase1}) implicitly entails a choice of phase.  That is, a given Calabi--Yau structure $(L, \omega_L, \Upsilon_L)$ belongs to an $S^1$-family of Calabi--Yau structures $(L, \omega_L, e^{i\theta}\Upsilon_L)$ for $\theta \in [0,2\pi)$, which in turn yields $\Spin(7)$-structures $\Phi_{(L, e^{i\theta})} := \frac{1}{2}\omega_L^2 + \mathrm{Re}(e^{i\theta}\Upsilon_L)$.  Clearly, $\Phi_L = \Phi_{(L, 1)}$. \\
\indent A separate phenomenon is that distinct Calabi--Yau structures can induce the same $\Spin(7)$-structure.  For example, the formula $\Phi_K = \omega_I^2 - \Phi_I$ used in the proof of Proposition \ref{prop:SU4Spin7A} can be read as saying that $\Phi_{(K,1)} = \Phi_{(I, -1)}$.
\end{rmk}

We note that the equivalence (a) $\Longleftrightarrow$ (b) in Corollary \ref{cor:Sp2Spin7relation} holds in every quaternionic dimension $m \geq1$. This follows from the definition of $H$-instantons, by observing that if $L \in \{I,J,K\}$:
$$
\mathfrak{su}_L(2m) = \{ A \in \mathfrak{so}(4m) | [A,L]=0, \text{Tr}_\C(A)=0\},
$$
and since $[A,I]= [A,J]K + J [A,K]$, and $\mathfrak{sp}(m) = \{ A \in \mathfrak{so}(4m) | [A,I]=[A,J]=[A,K]=0\}$, we get
$$
\mathfrak{sp}(m) = \mathfrak{su}_J(2m) \cap  \mathfrak{su}_K(2m).
$$

\section{The Calabi manifold $X = T^*\CP^2$} \label{sec:CalabiGeometry}

\indent \indent This section introduces the Calabi manifold $X = T^*\CP^2$, the central player of this work.  In $\S$\ref{sub:CalabiCohomOne}, we rapidly review the cohomogeneity-one structure of $X$.  To exploit this symmetry, we will make frequent use of the Maurer--Cartan form on $\SU(3)$, which we discuss in $\S$\ref{sub:MCForm}.  Finally, in $\S$\ref{sub:CalabiHKStructure}, we recall the Calabi metric and hyperk\"{a}hler structure on $X$.

\subsection{The cohomogeneity-one structure of $T^*\CP^2$} \label{sub:CalabiCohomOne}

\subsubsection{Cohomogeneity one manifolds}

\indent \indent Let $X$ be a connected Riemannian manifold, and let $G \leq \mathrm{Isom}(X)$ be a closed subgroup of the isometry group.

\begin{defn} ${}$
\begin{itemize}
\item We say $x \in X$ has \emph{larger orbit type than} (resp., \emph{the same orbit type as}) $y \in X$ if there exists $g \in G$ such that $G_x \leq G_{gy}$ (resp., $G_x = G_{gy}$).
\item A point $x \in X$ (or its orbit $Gx \subset X$) is called \emph{principal} if there exists a neighborhood $U \subset X$ of $x \in X$ such that $x$ has larger orbit type than every $y \in U$. 
\item A point $x \in X$ (or its orbit $Gx \subset X$) is called \emph{singular} if $\dim(Gx)$ is less than the dimension of a principal orbit.
\end{itemize}
\end{defn}

We let $X^* \subset X$ denote the union of principal orbits.  The \emph{principal orbit theorem} \cite[Theorem 3.82]{AB15} asserts that $X^* \subset X$ is open and dense, that $X^*/G$ is a connected manifold, and that any two principal orbits have the same orbit type. \\
\indent We say $X$ is \emph{cohomogeneity one} if its principal orbits have codimension one.  We will be concerned with cohomogeneity one manifolds that admit exactly one singular orbit $X_{\mathrm{sing}}$.  In this case, $X$ is non-compact and $X/G$ is diffeomorphic to $[0,\infty)$.  Moreover, there are exactly two orbit types.  Indeed, with respect to the projection $X \to X/G \approx [0, \infty)$, the principal orbits are preimages of points in $(0,\infty)$, whereas the singular orbit is the preimage of $\{0\}$. \\
\indent We let $K$ denote the isotropy group of the principal orbits, and let $H$ denote the isotropy group of the singular orbits.  Thus, each principal orbit is equivariantly diffeomorphic to $G/K$, and the singular orbit $X_{\mathrm{sing}}$ is equivariantly diffeomorphic to $G/H$.  The \emph{group diagram} of $X$ refers to the triple of Lie groups $K \leq H \leq G$ with their specified embeddings.  It is known that $H/K$ is diffeomorphic to a sphere.  Conversely, if $K \leq H \leq G$ are all compact and $H/K$ is a sphere, then there exists a cohomogeneity one manifold $X$ with a unique singular orbit whose group diagram is $K \leq H \leq G$.  For more information, we refer the reader to \cite{Berand1982, AB15}.


\subsubsection{The cohomogeneity-one structure of $T^*\CP^2$}

\indent \indent We now describe $X = T^*\CP^2$ as a cohomogeneity one manifold.  Recall that $G = \SU(3)$ acts transitively on $\CP^2$ with stabilizer $H = \U(2)$.  Explicitly, at the point $p_0 = [0 \colon 0 \colon 1]$, say, the isotropy group is
\begin{equation}\label{eq:U2}
H = \U(2) \cong \mathrm{S}(\U(2) \times \U(1)) = \left\{ \begin{pmatrix} A & 0 \\ 0 & \det(\overline{A}) \end{pmatrix} \colon A \in \U(2) \right\}\!.
\end{equation}
This $\SU(3)$-action on $\CP^2$ induces the following natural $\SU(3)$-action on $X = T^*\CP^2$.  An element $g \in \SU(3)$ acts on a covector $\alpha \in T^*_p\CP^2$ via
$$(g \cdot \alpha)(v) = \alpha( d(g^{-1})v), \ \ v \in T_{gp}\CP^2.$$
\indent We now compute the isotropy of the $\SU(3)$-action on a covector $\alpha \in T_{p_0}^*\CP^2$ at $p_0 = [0 \colon 0 \colon 1]$.  For $h \in H$ and a vector $v \in T_{p_0}\CP^2 \simeq \mathrm{span}_\C(e_1, e_2)$, we have $dh(v) = \det(A) Av$, so that
$$(h \cdot \alpha)(v) = \det(\overline{A})\, \alpha(A^*v).$$
Suppose that $h \cdot \alpha = \alpha$, and write $A = \begin{pmatrix} x & y \\ z & w \end{pmatrix}$.  For concreteness, let us work in the chart $\{z_3 \neq 0\}$ centered at $p_0 \in \CP^2$ via homogeneous coordinates $(w_1, w_2) = (z_1/z_3, z_2/z_3)$ and take $\alpha = dw_1$.  So, the condition $(h \cdot \alpha)(v) = \alpha(v)$ becomes $\det(\overline{A}) dw_1(A^*v) = \, dw_1(v)$.  Taking $v = \frac{\partial}{\partial w_2}$ yields $z = 0$.  Since $A \in \U(2)$, this implies $y = 0$ and $w = x^{-2}$.  Consequently, $A = \mathrm{diag}(x, x^{-2})$.  This shows that the isotropy group at a non-zero covector in $T^*_{p_0}\CP^2$ is
\begin{equation}\label{eq:U1}
K \cong \U(1) = \left\{ \begin{pmatrix} e^{i\theta} & 0 & 0 \\ 0 & e^{-2i\theta} & 0 \\ 0 & 0 & e^{i\theta} \end{pmatrix} \in \SU(3) \colon \theta \in [0,2\pi)  \right\}\!.
\end{equation}
\indent Thus, the geometric picture is as follows.   Equipping $\CP^2$ with the Fubini-Study metric, the $\SU(3)$-action on each sphere bundle $\{\alpha \in T^*\CP^2 \colon |\alpha| = r\}$ having $r > 0$ is transitive with stabilizer $K \cong \U(1)$.  These are the principal orbits, each equivariantly diffeomorphic to the Aloff--Wallach space $N_{1,-2} = \SU(3)/\U(1) = G/K$.  In particular, the principal locus $X^* = T^*\CP^2 - \CP^2$ is diffeomorphic to $(0,\infty) \times N_{1,-2}$.  On the other hand, the $\SU(3)$-action on the zero section is transitive with stabilizer $H \cong \U(2)$.  This is the unique singular orbit, equivariantly diffeomorphic to $\CP^2 = \SU(3)/\U(2) = G/H$.  

\begin{rmk} 
Our particular embeddings $K \leq H \leq \SU(3)$ depend on the base point.  Had we chosen $q_0 = [1 \colon 0 \colon 0]$, we would instead have $H' = \mathrm{S}(\U(1) \times \U(2))$ and $K' = \mathrm{diag}(e^{i\theta}, e^{i\theta}, e^{-2i\theta})$, and therefore the principal orbits would be equivariantly diffeomorphic to $N_{1,1}$.
By contrast, if we keep $H=\mathrm{S}(\U(2)\times\U(1))$ fixed and replace $K$ by $K'$, then $H/K'\cong\mathbb{RP}^3$, rather than $H/K \cong S^3$.
Thus, although $N_{1,1}$, $N_{1,-2}$ and $N_{-2,1}$ are $\SU(3)$-equivariantly diffeomorphic, their standard projections to the singular orbit $\SU(3)/H=\CP^2$ give an $\mathbb{RP}^3$-bundle in the first case and equivalent $S^3$-bundles in the latter two cases.
Collapsing the $\mathbb{RP}^3$ fibres at the $t=0$ end of
$(0,\infty)\times N_{1,1}$ yields an 8-dimensional orbifold admitting a $\Spin(7)$ metric; see \cite{CGLP02}.
\end{rmk}

\subsection{The Maurer--Cartan form of $\SU(3)$} \label{sub:MCForm}

\indent \indent For computations, we shall choose a basis $\{H, X_1, \ldots, X_7\}$ of $\mathfrak{su}(3)$ that is in some sense adapted to the geometry of the Aloff--Wallach space $N_{1,-2} = \SU(3)/\U(1)$.  To begin, let
\begin{equation} \label{eq:HBasis}
H = \begin{bmatrix} i & 0 & 0 \\ 0 & -2i & 0 \\ 0 & 0 & i \end{bmatrix}\!,
\end{equation}
and notice that $\mathfrak{u}(1) \cong \mathrm{span}(H) \subset \mathfrak{su}(3)$ is the Lie subalgebra corresponding to the embedding (\ref{eq:U1}).  Next, for $X,Y \in \mathfrak{su}(3)$, let $\langle X,Y \rangle := -\mathrm{tr}(XY)$ denote the positive-definite inner product given by the Killing form.  We now define, once and for all, a reductive complement $\mathfrak{m} \subset \mathfrak{su}(3)$ by the $\langle \cdot, \cdot \rangle$-orthogonal decomposition
$$\mathfrak{su}(3) = \mathfrak{u}(1) \oplus \mathfrak{m}.$$
Now, define a basis $\{X_1, \ldots, X_7\}$ of $\mathfrak{m}$ as follows: 
\begin{align}  \label{eq:XBasis}
X_4 & = \begin{bmatrix}
i & 0 & 0 \\
0 & 0 & 0 \\
0 & 0 & -i \end{bmatrix}, & X_1 & = \begin{bmatrix}
0 & 1 & 0 \\
-1 & 0 & 0 \\
0 & 0 & 0 \end{bmatrix}, & X_2 & =  \begin{bmatrix}
0 & 0 & 0 \\
0 & 0 & -1 \\
0 & 1 & 0 \end{bmatrix}, & X_3 & = \begin{bmatrix}
0 & 0 & -1 \\
0 & 0 & 0 \\
1 & 0 & 0 \end{bmatrix}, \\
& &  X_5 & = \begin{bmatrix}
0 & i & 0 \\
i & 0 & 0 \\
0 & 0 & 0 \end{bmatrix}, & X_6 & =  \begin{bmatrix}
0 & 0 & 0 \\
0 & 0 & i \\
0 & i & 0 \end{bmatrix}, &
 X_7 & =  \begin{bmatrix}
0 & 0 & i \\
0 & 0 & 0 \\
i & 0 & 0 \end{bmatrix}\!. \notag
\end{align}
One can check that $\{\frac{1}{\sqrt{6}}H, \frac{1}{\sqrt{2}}X_1, \ldots, \frac{1}{\sqrt{2}}X_7\}$ is an $\langle \cdot, \cdot \rangle$-orthonormal basis of $\mathfrak{su}(3)$.  Moreover, the $\U(1)$-action on $\mathfrak{m} = \mathrm{span}\{X_1, \ldots, X_7\}$ decomposes into $\U(1)$-irreducible submodules as
\begin{equation} \label{eq:MDecomp}
\mathfrak{m} = (X_4) \oplus (X_3) \oplus (X_7) \oplus (X_1, X_5) \oplus (X_2, X_6),
\end{equation}
where we are abbreviating $(X_4) = \mathrm{span}(X_4)$, and so on.  \\

\indent Now, let $\eta \in \Omega^1(\SU(3); \mathfrak{su}(3))$ denote the Maurer--Cartan form of $\SU(3)$.  Expressing $\eta$ in terms of the basis $\{H,X_1, \ldots, X_7\}$ defined above, we may write
\begin{align*}
\eta & = \zeta \otimes H + \kappa \otimes X_4 + \sigma_1 \otimes X_1 + \sigma_2 \otimes X_5 \\
& \ - \nu_1 \otimes X_3 + \nu_2 \otimes X_7 - \mu_1 \otimes X_2 + \mu_2 \otimes X_6,
\end{align*}
for some $1$-forms $\zeta, \kappa, \nu_1, \nu_2, \sigma_1, \sigma_2, \mu_1, \mu_2 \in \Omega^1(\SU(3))$.  Defining complex $1$-forms $\nu := \nu_1 + i\nu_2$ and $\sigma := \sigma_1 + i\sigma_2$ and $\mu := \mu_1 + i\mu_2$, we may write $\eta$ more succinctly as
$$\eta = \begin{bmatrix}
i(\zeta + \kappa) & \sigma & \nu \\
-\overline{\sigma} & -2i\zeta & \mu \\
-\overline{\nu} & -\overline{\mu} & i(\zeta - \kappa) \end{bmatrix}\!.$$
Expanding the Maurer--Cartan equation $d\eta_{AB} = -\eta_{AC} \wedge \eta_{CB}$ yields the following \emph{structure equations}:
\begin{align*}
d\sigma_1 & =3\zeta\wedge\sigma_2 - \mu_1\wedge\nu_1 + \kappa\wedge\sigma_2 - \mu_2\wedge\nu_2,  & 
d\nu_1 & = -\sigma_1\wedge\mu_1 + 2\kappa\wedge\nu_2 + \sigma_2\wedge\mu_2, \\
d\sigma_2 & = -3\zeta\wedge\sigma_1 + \sigma_1\wedge\kappa - \mu_1\wedge\nu_2 - \nu_1\wedge\mu_2, & 
d\nu_2 &=-\sigma_1\wedge\mu_2 + \mu_1\wedge\sigma_2 + 2\nu_1\wedge\kappa, \\
d\mu_1 & = -3\zeta\wedge\mu_2 + \sigma_1\wedge\nu_1 + \kappa\wedge\mu_2 + \sigma_2\wedge\nu_2, & 
d\kappa & =  -\sigma_1\wedge\sigma_2 - \mu_1\wedge\mu_2 - 2\,\nu_1\wedge\nu_2, \\
d\mu_2 & = 3\zeta\wedge\mu_1 + \sigma_1\wedge\nu_2 + \mu_1\wedge\kappa + \nu_1\wedge\sigma_2,  & 
d\zeta & = -\sigma_1\wedge\sigma_2 + \mu_1\wedge\mu_2.
\end{align*}

\subsection{The Calabi hyperk\"{a}hler structure on $T^*\CP^2$} \label{sub:CalabiHKStructure}

\indent \indent We now equip $X = T^*\CP^2$ with its Calabi hyperk\"{a}hler structure.  Our discussion here is drawn from \cite{cvetivc2001hyper}, though our notation and scale factor conventions are somewhat different. \\
\indent Let us regard the principal locus $X^* = T^*\CP^2 - \CP^2$ as $(0,\infty) \times N_{1,-2}$, and use the coordinate $t$ on the $(0,\infty)$ factor.  Every $\SU(3)$-invariant Riemannian metric $g$ on $X^*$ can be expressed (after pulling back to $(0,\infty) \times \SU(3)$) as
\begin{equation} \label{eq:Metric1}
g = dt^2 + a^2\left( \sigma_1^2 + \sigma_2^2 \right) + b^2 \left( \mu_1^2 + \mu_2^2 \right) + c_1^2 \nu_1^2 + c_2^2\nu_2^2 + f^2 \kappa^2,
\end{equation}
for some functions $a(t), b(t), c_1(t), c_2(t), f(t)$.  Now, define a new coordinate $r \in (1,\infty)$ by requiring
\begin{equation} \label{eq:dtdr}
dt = \sqrt{ \frac{r^4}{r^4 - 1} }\,dr.
\end{equation}
The \emph{Calabi metric} is the $\SU(3)$-invariant metric (\ref{eq:Metric1}) with coefficient functions
\begin{align} \label{eq:Metric2}
a & =  \sqrt{\frac{1}{2}(r^2 - 1)}, & b & =  \sqrt{\frac{1}{2}(r^2+1)}, & c := c_1 = c_2 & = r, & f & =  \frac{1}{r}\sqrt{r^4 - 1}.
\end{align}
It is well-known that the Calabi metric extends across the zero section, and is a complete Riemannian metric on all of $T^*\CP^2$ \cite{dancer1997hyperkahler}.  By construction, the set of covectors $\{e^0, e^1, \ldots, e^7\}$ defined by
\begin{align} \label{eq:Coframe}
e^0 & = dt, & e^1 & = a\sigma_1, & e^3 & = -b\mu_1, & e^5 & = c\nu_1, & e^7 & = f\kappa, \\
 &   & e^2 & = a\sigma_2, & e^4 & = b\mu_2, & e^6 & = c\nu_2, &  \notag
\end{align}
is $g$-orthonormal.  Using these, we define the \emph{hyperk\"{a}hler $2$-forms}
\begin{align}
\omega_I & = (-e^0 \wedge e^7 + e^5 \wedge e^6) + (e^1 \wedge e^2 - e^3 \wedge e^4), \notag \\
\omega_J & = (-e^0 \wedge e^6 - e^5 \wedge e^7) + (-e^2 \wedge e^3 + e^1 \wedge e^4), \label{eq:CalabiHK} \\
\omega_K & = (e^0 \wedge e^5 - e^6 \wedge e^7) + (e^2 \wedge e^4 + e^1 \wedge e^3). \notag
\end{align}
One can check that $\omega_I, \omega_J, \omega_K$ are non-degenerate, closed $2$-forms that endow $X$ with a hyperk\"{a}hler structure.  Note that the orientation on $X$ is given by $\vol_X = e^0 \wedge e^1 \wedge e^2 \wedge e^3 \wedge e^4 \wedge e^5 \wedge e^6 \wedge e^7$. \\
\indent Finally, we comment on the asymptotic behavior of the Calabi metric. 

\begin{defn}
	A Riemannian manifold \( X\) is said to be \textit{asymptotically conical} with \textit{rate} \( \nu < 0\) and compact asymptotic link \( \left( \Sigma, g_{\Sigma} \right)\) if there exists a compact set \( K \subset X\), a positive number \( R\) and a diffeomorphism 
	\begin{equation*}
		\Psi: \left( R,\infty \right) \times \Sigma \xrightarrow{\sim} X - K
	\end{equation*}
	such that for any \( k \geq 0\), we have, for \( r \rightarrow \infty\), 
	\begin{equation*}
		\abs*{\nabla^k_{g_C}\left( \Psi^* g_X - g_C \right)}_{g_C} = O\left( r^{\nu - k} \right).
	\end{equation*}
	Here, \( g_C\) is the cone metric on \( C(\Sigma)\) given by \( g_C = dr^2 + r^2 g_{\Sigma}\).
	\label{def:AC}
\end{defn}

Note that we may rewrite $g$ as follows:
$$g = \frac{r^4}{r^4-1}\,dr^2 + r^2 \left( \frac{r^2 - 1}{2r^2}\left( \sigma_1^2 + \sigma_2^2 \right) + \frac{r^2+1}{2r^2} \left( \mu_1^2 + \mu_2^2 \right) + (\nu_1^2 + \nu_2^2) + \frac{r^4 - 1}{r^4}\,\kappa^2\right)\!.$$
As $r \to \infty$, the Calabi metric $g$ appears to approach the cone metric on $\mathrm{C}(N_{1,-2}) = (0,\infty) \times N_{1,-2}$ defined by
$$g_C := dr^2 + r^2 \left( \frac{1}{2}\!\left( \sigma_1^2 + \sigma_2^2 \right) + \frac{1}{2}\! \left( \mu_1^2 + \mu_2^2 \right) + (\nu_1^2 + \nu_2^2) + \kappa^2\right)\!.$$
To be precise:

\begin{rmk}
The Calabi metric $g$ is an asymptotically conical Riemannian metric with rate $\nu = -2$ and asymptotic link $(N_{1,-2}, g_N)$, where 
$
	g_N = \frac{1}{2}\left( \sigma_1^2 + \sigma_2^2 \right) + \frac{1}{2}\left( \mu_1^2 + \mu_2^2 \right) + \left( \nu_1^2 + \nu_2^2 \right) + \kappa^2.
$
\end{rmk}

\indent By way of comparison, we remark that the Stenzel metric on $T^*S^4$ has rate $\nu = -\frac{8}{3}$, and the Bryant-Salamon metric on $\Spinor(S^4)$ has rate $\nu = -\frac{10}{3}$.  See \cite{conlon2013asymptotically}, \cite{ghosh2024deformation}.

\section{Invariant bundles and connections} \label{sec:InvarBundConn}

\indent \indent This section reviews the basics of homogeneous bundles and invariant connections.  In $\S$\ref{sub:HomBundle}, we define homogeneous bundles $P \to G/L$ over homogeneous spaces and explain how to construct them.  Then, in $\S$\ref{sub:BundleExtension}, we consider bundles $E \to X^*$ over the principal locus $X^*$ of an arbitrary cohomogeneity one manifold $X = X^* \cup X_{\mathrm{sing}}$ with unique singular orbit $X_{\mathrm{sing}}$ that restrict to homogeneous bundles on the principal orbits.  The main result (Proposition \ref{prop:BundleExtension}) is an easily checked necessary and sufficient condition for such bundles $E \to X^*$ to extend across the singular orbit.  Although this extension result is well-known to experts, we were unable to locate a proof in the literature.  Section \ref{sub:InvariantP} recalls Wang's theorem on invariant connections on homogeneous bundles $P \to G/L$, and $\S$\ref{sub:InvariantE} discusses invariant connections on bundles $E \to X^*$ over principal loci.

\subsection{Homogeneous principal bundles} \label{sub:HomBundle}

\indent \indent We begin by recalling the concept of a ``homogeneous $S$-bundle". Let $G/L$ be a homogeneous space, where $G$ is a compact connected Lie group and $L \leq G$ is a closed subgroup.  In actual practice, we will take $G = \SU(3)$ and $L = \U(1)$ or $\U(2)$, where the embeddings are given by (\ref{eq:U1})-(\ref{eq:U2}), so that $G/L = \SU(3)/\U(1) = N_{1,-2}$ or $G/L = \SU(3)/\U(2) = \CP^2$.  Let $S$ be a compact Lie group, which will serve as our desired structure group.

\begin{defn} A \emph{homogeneous $S$-bundle} over $G/L$ is a principal $S$-bundle $\pi \colon P \to G/L$ whose total space admits a left $G$-action that lifts the $G$-action on $G/L$.  Two homogeneous $S$-bundles $P_1 \to G/L$ and $P_2 \to G/L$ are \emph{isomorphic} if there exists a principal bundle isomorphism $P_1 \to P_2$ that is $G$-equivariant.
\end{defn}

\indent Note that if $P \to G/L$ is a homogeneous $S$-bundle, then $P$ admits both a left $G$-action as well as a right $S$-action. Together, the $(G \times S)$-action on $P$ given by $(g,s) \cdot p = gps^{-1}$ is transitive, and the stabilizer is isomorphic to $L$.   In fact, fixing a base point $p \in P$, there exists a unique Lie group homomorphism $\psi \colon L \to S$ such that $p \psi(\ell) = \ell p$, and the stabilizer at $p$ is given by $\mathrm{Stab}_{G \times S}(p) = (\mathrm{Id} \times \psi)(L) = \{ (\ell, \psi(\ell)) \colon \ell \in L\} \leq G \times S$.  Consequently, the smooth manifold $P$ is itself a homogeneous space $P \cong (G \times S)/L$.

\indent It turns out that this process can be reversed.  Indeed, the following standard result classifies homogeneous $S$-bundles over $G/L$ in terms of Lie group homomorphisms $\psi \colon L \to S$.  Say that two homomorphisms $\psi_1, \psi_2 \colon L \to S$ are \emph{(element) conjugate} if there exists $s \in S$ such that $\psi_2 = s \psi_1 s^{-1}$.

\begin{lem}[Construction lemma] \label{lem:Construction} For each Lie group homomorphism $\psi \colon L \to S$, let
$$P_\psi := G \times_\psi S = \frac{G \times S}{\sim}, \ \text{ where } (g,s) \sim (g \ell, s \psi(\ell)) \ \text{ for } (g,s) \in G \times S \text{ and } \ell \in L.$$
Then $\pi_\psi \colon P_\psi \to G/L$ via $\pi_\psi([g,s]_\psi) = gL$ is a homogeneous $S$-bundle.  Moreover, every homogeneous $S$-bundle over $G/L$ arises in this way.  Finally, $P_{\psi_1}$ is isomorphic to $P_{\psi_2}$ as homogeneous $S$-bundles if and only if $\psi_1$ and $\psi_2$ are element-conjugate.
\end{lem}

\subsection{The bundle extension lemma} \label{sub:BundleExtension}

\indent \indent Our goal is to study gauge-theoretic objects on suitable $S$-bundles over $X = X^* \cup X_{\mathrm{sing}}$, where $X = T^*\CP^2$ and $X_{\mathrm{sing}} = \CP^2$ is the zero section.   For a given structure group $S$, we will construct the ``suitable $S$-bundles" over $T^*\CP^2$ via the following process:
\begin{enumerate}
\item For each homomorphism $\lambda \colon \U(1) \to S$, the construction lemma \ref{lem:Construction} yields a homogeneous $S$-bundle $P_\lambda \to N_{1,-2}$.  Letting $\mathrm{pr} \colon X^* \approx (0,\infty) \times N_{1,-2} \to N_{1,-2}$ denote the projection, we obtain an $S$-bundle $E_\lambda := \mathrm{pr}^*P_\lambda \to X^*$ over the principal locus via pullback.
\item For each homomorphism $\chi \colon \U(2) \to S$, the construction lemma \ref{lem:Construction} yields a homogeneous $S$-bundle $Q_\chi \to \CP^2$.  Letting $\Pi \colon T^*\CP^2 \to \CP^2$ denote the projection, we obtain an $S$-bundle $\widetilde{E}_\chi := \Pi^*Q_\chi \to X$ over the entire manifold via pullback.
\item Finally, for ``matching pairs" of homomorphisms $(\lambda, \chi)$, the bundle $\widetilde{E}_\chi \to X$ restricts to the principal locus to agree with $E_\lambda \to X^*$, i.e.\ $\widetilde{E}_{\chi}$ is a $G$-equivariant extension of $E_\lambda$.
Here, a $G$-equivariant extension means that $\widetilde{E}_{\chi}|_{X^*}$ and $E_\lambda$ are $G$-equivariantly isomorphic as principal $S$-bundles over $X^*$.
\end{enumerate}
\indent In this section, we provide an easily-checked necessary and sufficient criterion for determining whether two such homomorphisms $\lambda$ and $\chi$ ``match," and therefore whether $\widetilde{E}_\chi \to X$ is a $G$-equivariant extension of $E_\lambda \to X^*$.  This criterion appears to be well-known to experts, but we provide a proof for completeness.  Although we are exclusively interested in the case where the group diagram $G \geq H \geq K$ is $\SU(3) \geq \U(2) \geq \U(1)$, we will state the result in greater generality.

\begin{setup} \label{setup:cohom-one} Let $X$ be a cohomogeneity-one smooth manifold with a unique singular orbit $X_{\mathrm{sing}} \subset X$, and let $X^* = X - X_{\mathrm{sing}}$ denote the principal locus.  Suppose that $X$ has group diagram $G \geq H \geq K$, where $G,H,K$ are compact connected Lie groups.  
\begin{itemize}
\item Let $\iota \colon K \hookrightarrow H$ denote inclusion, and let $p \colon G/K \to G/H$ denote the natural projection.
\item Let $\mathrm{pr} \colon X^* \approx (0,\infty) \times (G/K) \to G/K$ denote the projection.
\item Let $\Pi \colon X \to G/H$ denote the natural extension of $p \circ \mathrm{pr} \colon X^* \to G/H$.
\end{itemize}
\end{setup}

\begin{prop}[Bundle extension lemma] \label{prop:BundleExtension} Let $\lambda \colon K \to S$ and $\chi \colon H \to S$ be homomorphisms, and let $\pi_\lambda \colon P_\lambda \to G/K$ and $\pi_\chi \colon Q_\chi \to G/H$ denote the associated homogeneous $S$-bundles.  The following are equivalent:
\begin{enumerate}[(1)]
\item The map $\lambda$ is element conjugate to $\chi \circ \iota$.
\item We have $P_\lambda \cong p^*Q_\chi$.
\item The bundle $\widetilde{E}_{\chi}:=\Pi^*Q_\chi\to X$ is a $G$-equivariant extension of $E_\lambda:=\mathrm{pr}^*P_\lambda\to X^*$.
\end{enumerate}
\end{prop}
\begin{proof}
Before beginning, observe that
\begin{align*}
p^*Q_\chi & = \{ (g_1K, [g_2, s]_\chi) \in G/K \times Q_\chi \colon g_1,g_2 \in G, s \in S,  g_1H = g_2H \}  \\
& = \{ (gK, [g,s]_\chi) \in G/K \times Q_\chi \colon g \in G, s \in S \}.
\end{align*}
In particular, we see that $p^*Q_\chi \to G/K$ is a homogeneous $S$-bundle.  Next, consider the following two maps from $G \times S$:
\begin{align*}
v \colon G \times S & \to p^*Q_\chi; & q_{\chi \circ \iota} \colon G \times S & \to P_{\chi \circ \iota};  \\
v(g,s) & = (gK, [g,s]_\chi) & q_{\chi \circ \iota}(g,s) & = [g,s]_{\chi \circ \iota}.
\end{align*}
One can check that the $v$-fibers and $q_{\chi \circ \iota}$-fibers are given by:
\begin{align*}
v^{-1}( aK, [a,t]_\chi) & = \{ (ak, t\chi(\iota(k))) \in G \times S \colon k \in K \} = q_{\chi \circ \iota}^{-1}( [a,t]_{\chi \circ \iota}).
\end{align*}
Since $v \colon G \times S \to p^*Q_\chi$ and $q_{\chi \circ \iota} \colon G \times S \to P_{\chi \circ \iota}$ are $G$-equivariant surjective submersions that are constant on each other's fibers, there exists unique $G$-equivariant diffeomorphism $F \colon P_{\chi \circ \iota} \to p^*Q_\chi$ for which $v = F \circ q_{\chi \circ \iota}$.  It is straightforward to check that $F$ is, in fact, a principal bundle isomorphism, so that $p^*Q_\chi \cong P_{\chi \circ \iota}$.
$$\begin{tikzcd}
                                              & G \times S \arrow[ld, "q_{\chi \circ \iota}"'] \arrow[rd, "v"] &                      \\
P_{\chi \circ \iota} \arrow[rd] \arrow[rr, "F"', dashed] &                                                     & p^*Q_\chi \arrow[ld] \\
                                              & G/K                                                 &                     
\end{tikzcd}$$
\indent (1) $\iff$ (2).  By the construction lemma \ref{lem:Construction}, the map $\lambda$ is element conjugate to $\chi \circ \iota$ if and only if $P_\lambda$ is isomorphic to $P_{\chi \circ \iota} \cong p^*Q_\chi$. \\
\indent (2) $\iff$ (3).
Let $\widetilde{E}_{\chi}:=\Pi^*Q_\chi \rightarrow X$, let
$j_{\mathrm{prin}}\colon X^*\hookrightarrow X$ denote the inclusion, and observe that
\begin{equation}
\label{eq:BundlePullbacks}
j_{\mathrm{prin}}^*\widetilde{E}_{\chi}
= j_{\mathrm{prin}}^*\Pi^*Q_\chi
\cong (\Pi\circ j_{\mathrm{prin}})^*Q_\chi
= (p\circ\mathrm{pr})^*Q_\chi
\cong \mathrm{pr}^*(p^*Q_\chi).
\end{equation}
These are canonical $G$-equivariant isomorphisms.
Now, if (2) holds, then \eqref{eq:BundlePullbacks} gives
$
j_{\mathrm{prin}}^*\widetilde{E}_{\chi}
\cong \mathrm{pr}^*(p^*Q_\chi)
\cong \mathrm{pr}^*P_\lambda
=E_\lambda
$
$G$-equivariantly, which is (3). Conversely, if (3) holds, then \eqref{eq:BundlePullbacks}
gives a $G$-equivariant isomorphism
$
\mathrm{pr}^*P_\lambda
=E_\lambda
\cong j_{\mathrm{prin}}^*\widetilde{E}_{\chi}
\cong \mathrm{pr}^*(p^*Q_\chi).
$
Letting $j_1\colon G/K\hookrightarrow X^*$ denote a $G$-equivariant right-inverse of $\mathrm{pr}\colon X^*\to G/K$, we have
$
P_\lambda
\cong j_1^*(\mathrm{pr}^*P_\lambda)
\cong j_1^*(\mathrm{pr}^*(p^*Q_\chi))
\cong p^*Q_\chi,
$
which is (2).
\end{proof}

\indent For use in $\S$\ref{sub:InvariantE}, we briefly consider the principal $S$-bundles $\widehat{\pi}_\lambda \colon E_\lambda \to X^*$ in greater detail.  For this, let $\phi \colon (0,\infty) \times G/K \to X^*$ be the identification map, so that every $x \in X^*$ is of the form $x = \phi(t, gK)$ for some $t \in (0,\infty)$ and $gK \in G/K$.  Essentially by definition, we have
$$E_\lambda = \mathrm{pr}^*P_\lambda = \left\{ ( \phi( t, \pi_\lambda(p) ), p) \in X^* \times P_\lambda \colon t \in (0,\infty), p \in P_\lambda\right\}\!.$$
Note that the total space $E_\lambda$ admits a left $G$-action via $g \cdot (x,p) = ( gx, gp)$ for $g \in G$, and the projection map $\widehat{\pi}_\lambda \colon E_\lambda \to X^*$ is $G$-equivariant with respect to this action.  We observe:

\begin{prop} \label{prop:Product} The principal $S$-bundle $\widehat{\pi}_\lambda \colon E_\lambda \to X^*$ is $G$-equivariantly isomorphic to the product bundle $(\mathrm{Id}, \pi_\lambda) \colon (0,\infty) \times P_\lambda \to (0,\infty) \times (G/K)$.
\end{prop}

\begin{proof} The map $F \colon (0,\infty) \times P_\lambda \to E_\lambda$ given by $F(t,p) = ( \phi(t, \pi_\lambda(p)), p)$ is a $G$-equivariant diffeomorphism.  Moreover, it is clear that $\widehat{\pi}_\lambda(F(t,p)) = \phi(t, \pi_\lambda(p))$ for all $(t,p) \in (0,\infty) \times P_\lambda$, showing that $\widehat{\pi}_\lambda \circ F = \phi \circ (\mathrm{Id}, \pi_\lambda)$, which means that $F$ is a bundle isomorphism.
\end{proof}

\subsection{Invariant connections on $P_\lambda$}  \label{sub:InvariantP}

\indent \indent We now turn our attention to invariant connections on homogeneous bundles.  Let $P \to G/K$ be a homogeneous $S$-bundle, where $G$ is a compact connected Lie group, $K \leq G$ is a closed subgroup, and $S$ is a compact Lie group.  Recall that the smooth manifold $P$ admits both a left $G$-action as well as a right $S$-action.

\begin{defn} A connection $A \in \Omega^1(P; \mathfrak{s})$ is \emph{$G$-invariant} (or just \emph{invariant}) if $L_g^*A = A$ for all $g \in G$, where $L_g \colon P \to P$ denotes the left $G$-action.
\end{defn}

\indent Note that any connection $A \colon TP \to \mathfrak{s}$, invariant or otherwise, is always equivariant with respect to the right $S$-action on $P$ and the adjoint $S$-action on $\mathfrak{s}$.  The following well-known result classifies invariant connections on homogeneous bundles.

\begin{thm}[Wang's theorem \cite{wang1958invariant}] Let $P_\lambda \to G/K$ be a homogeneous $S$-bundle associated to a homomorphism $\lambda \colon K \to S$.  Let $\mathfrak{m} \subset \mathfrak{g}$ be a reductive complement, so that $\mathfrak{g} = \mathfrak{m} \oplus \mathfrak{k}$.  There is a bijective correspondence between:
\begin{enumerate}
\item $G$-invariant connections on $P_\lambda$, and
\item $K$-equivariant linear maps $\Lambda \colon \mathfrak{m} \to \mathfrak{s}$.
\end{enumerate}
By ``$K$-equivariant," we mean that $\Lambda \circ \mathrm{Ad}_k = \mathrm{Ad}_{\lambda(k)} \circ \Lambda$ for all $k \in K$.
\end{thm}

\begin{defn} Let $P_\lambda \to G/K$ be a homogeneous $S$-bundle.  Fix a reductive complement $\mathfrak{m} \subset \mathfrak{g}$.
\begin{itemize}
\item Let $\alpha_{\lambda, \Lambda} \in \Omega^1(P_\lambda; \mathfrak{s})$ denote the $G$-invariant connection corresponding to the $K$-equivariant linear map $\Lambda \colon \mathfrak{m} \to \mathfrak{s}$.
\item The \emph{canonical connection} on $P_\lambda$ is the $G$-invariant connection $\alpha_{\lambda, 0} \in \Omega^1(P_\lambda; \mathfrak{s})$ corresponding to the zero map $\Lambda = 0$.  It may be characterized as the $G$-invariant connection having $\mathrm{Ker}(\alpha_{\lambda, 0}) = \mathfrak{m}$.
\end{itemize}
\end{defn}

\indent Now, let $\Psi \colon G \to P_\lambda$ denote the composition $G \hookrightarrow G \times S \to P_\lambda$.  Given a connection $\alpha \in \Omega^1(P_\lambda; \mathfrak{s})$, its curvature $F_\alpha \in \Omega^2(P_\lambda; \mathfrak{s})$ pulls back to
\begin{equation} \label{eq:CurvatureGeneral}
\Psi^*F_\alpha = d\widehat{\alpha} + \frac{1}{2}[\widehat{\alpha} \wedge \widehat{\alpha}] \in \Omega^2(G; \mathfrak{s}),
\end{equation}
where here $\widehat{\alpha} = \Psi^*\alpha \in \Omega^1(G; \mathfrak{s})$.  Moreover, if $\Lambda \colon \mathfrak{m} \to \mathfrak{s}$ is a $K$-equivariant linear map, then the left-invariant extension of $\Lambda + d\lambda|_e \colon \mathfrak{m} \oplus \mathfrak{k} \to \mathfrak{s}$ on $G$ is the pullback $\widehat{\alpha}_{\lambda, \Lambda} = \Psi^*\alpha_{\lambda, \Lambda} \in \Omega^1(G; \mathfrak{s})$ of a $G$-invariant connection on $P_\lambda$.  Conversely, every $G$-invariant connection $\alpha_{\lambda, \Lambda} \in \Omega^1(P_\lambda, \mathfrak{s})$ pulls back to such a left-invariant extension. \\
\indent For the purposes of computation, we shall follow standard practice and identify connections $\alpha$ on $P_\lambda$ and their curvatures $F_\alpha$ with their respective pullbacks $\widehat{\alpha} = \Psi^*\alpha$ and $\Psi^*F_\alpha$ to $G$.  In other words, we shall suppress $\Psi^*$ from the notation, and use the same symbol $\alpha$ to denote both a connection on $P_\lambda$ as well as its pullback to $G$.

\subsection{Invariant connections on $E_\lambda$} \label{sub:InvariantE}

\indent \indent We now discuss invariant connections on the bundles $E_\lambda \to X^*$.  We continue to work in the context of Setup \ref{setup:cohom-one}.  That is, $X$ is a cohomogeneity-one smooth manifold with a unique singular orbit $X_{\mathrm{sing}}$ and group diagram $G \geq H \geq K$, and $\mathrm{pr} \colon X^* \approx (0,\infty) \times G/K \to G/K$ denotes the projection. \\
\indent Let $\lambda \colon K \to S$ be a homomorphism, let $\pi_\lambda \colon P_\lambda \to G/K$ be the associated homogeneous $S$-bundle, and set $E_\lambda := \mathrm{pr}^*P_\lambda \to X^*$ as in $\S$\ref{sub:BundleExtension}.  Recall that $E_\lambda \subset X^* \times P_\lambda$ admits a left $G$-action by restricting the diagonal action.  A connection $A \in \Omega^1(E_\lambda; \mathfrak{s})$ is called \emph{$G$-invariant} (or just \emph{invariant}) if $L_g^*A = A$ for all $g \in G$, where $L_g \colon E_\lambda \to E_\lambda$ denotes the left $G$-action. \\
\indent Let $A \in \Omega^1(E_\lambda; \mathfrak{s})$ denote an $\mathfrak{s}$-valued $1$-form.  By Proposition \ref{prop:Product}, we may identify the bundle $E_\lambda \to X^*$ with the product bundle $(0,\infty) \times P_\lambda \to (0,\infty) \times (G/K)$.  Letting $p_2 \colon E_\lambda \approx (0,\infty) \times P_\lambda \to P_\lambda$ denote the projection map, and letting $t$ denote the coordinate on $(0,\infty)$, we can write
\begin{equation}
A = \gamma\,dt + \widetilde{\alpha}(t),
\end{equation}
where $\gamma \in \Omega^0(E_\lambda; \mathfrak{s})$ and each $\widetilde{\alpha}(t) = p_2^*(\alpha(t))$ is the pullback of some $1$-form $\alpha(t) \in \Omega^1(P_\lambda; \mathfrak{s})$.  Note that for a given $A \in \Omega^1(E_\lambda; \mathfrak{s})$, the vector-valued forms $\gamma$ and $\alpha(t)$ are uniquely determined.  Moreover, if $A$ is a connection on $E_\lambda$, then each $\alpha(t)$ is a connection on $P_\lambda$.

\begin{defn} A connection $A \in \Omega^1(E_\lambda; \mathfrak{s})$ is said to be \emph{in temporal gauge} if $\gamma = 0$.
\end{defn}

\indent Note that if $A = \widetilde{\alpha}(t)$ is a $G$-invariant connection on $E_\lambda$ in temporal gauge, then each $\alpha(t)$ is a $G$-invariant connection on $P_\lambda$.  Conversely, if each $\alpha(t)$ is a $G$-invariant connection on $P_\lambda$, then $A := p_2^*(\alpha(t))$ is a $G$-invariant connection on $E_\lambda$ in temporal gauge.  \\
\indent The following well-known proposition assures us that, after a gauge transformation, every connection on $E_\lambda$ can be put in temporal gauge.

\begin{prop}[Temporal gauge] Let $A = \gamma\,dt + \widetilde{\alpha}(t)$ be a connection on $E_\lambda$.  Then there exists a gauge transformation $u \colon E_\lambda \to S$ such that $u \cdot A$ is in temporal gauge, and hence $u \cdot A = \widetilde{\rho}(t)$ for some connection forms $\rho(t)$ on $P_\lambda$. \\
\indent Moreover, if $A$ is $G$-invariant, then $u$ can be chosen such that $u \cdot A$ is $G$-invariant, and hence each $\rho(t)$ is a $G$-invariant connection on $P_\lambda$.
\end{prop}

\begin{proof} Since this is a standard fact, we will simply provide a sketch.  Regard $E_\lambda \cong (0,\infty) \times P_\lambda$, let $u \colon (0,\infty) \times P_\lambda \to S$ be a gauge transformation, and write $u \cdot A = \xi\,dt + \widetilde{\rho}(t)$.  Then $\xi = u\gamma u^{-1} + u(\partial_t u^{-1}) = u \gamma u^{-1} - (\partial_t u)u^{-1}$.  Thus, the temporal gauge condition $\xi = 0$ is equivalent to $u$ solving the ODE $\partial_t u(t,p) = u(t,p) \gamma(t,p)$.  Fixing $t_0 \in (0, \infty)$ and imposing the initial condition $u(t_0, \cdot) = \mathrm{Id}_S$, existence and uniqueness of $u$ follows from standard ODE theory.   \\
\indent Suppose $A$ is $G$-invariant.  For each $g \in G$, let $v_g \colon (0,\infty) \times P_\lambda \to S$ be $v_g(t,p) = u(t, g^{-1}p)$.  Using the $G$-invariance of $A$, one can show that $v_g$ solves the initial-value problem given by $\partial_t v_g(t,p) = v_g(t,p) \gamma(t,p)$ and $v_g(t_0, \cdot) = \mathrm{Id}_S$.  Thus, uniqueness implies that $v_g = u$, which means that $u$ is $G$-invariant, and hence $u \cdot A$ is $G$-invariant.
\end{proof}

\indent Writing $\widetilde{\alpha}_t = \widetilde{\alpha}(t)$ for brevity, let $A = \widetilde{\alpha}(t)$ be a connection on $E_\lambda$ in temporal gauge.  Then its curvature can be calculated as follows:
\begin{align*}
F_A = d\widetilde{\alpha}_t + \frac{1}{2}[ \widetilde{\alpha}_t \wedge \widetilde{\alpha}_t ] & = dt \wedge \partial_t \widetilde{\alpha}_t  + p_2^*(d\alpha_t) + \frac{1}{2}[ \widetilde{\alpha}_t \wedge \widetilde{\alpha}_t ] \\
& = dt \wedge \partial_t \widetilde{\alpha}_t + p_2^*F_{\alpha_t}.
\end{align*}
We will often abuse notation by omitting subscripts and pullbacks, writing the above formula as simply
\begin{equation} \label{eq:CurvatureFA}
F_A = dt \wedge \partial_t \alpha + F_\alpha.
\end{equation}

\section{Invariant instantons on $S^1$-bundles over $T^*\CP^2$} \label{sec:InvarInstS1}

\indent \indent We are at last in a position to study gauge-theoretic objects over $X = T^*\CP^2$.  For the remainder of this work, we let $X = T^*\CP^2$ equipped with the Calabi metric (\ref{eq:Metric1})-(\ref{eq:Metric2}) and hyperk\"{a}hler structure (\ref{eq:CalabiHK}).  Recall from $\S$\ref{sub:CalabiCohomOne} that $X$ is cohomogeneity-one with group diagram $G \geq H \geq K$ being $\SU(3) \geq \U(2) \geq \U(1)$, where the embeddings are given by (\ref{eq:U1})-(\ref{eq:U2}).  In particular, the inclusion map $\iota \colon \U(1) \hookrightarrow \U(2)$ is given by
\begin{equation} \label{eq:KHInclusion}
\iota(e^{i\theta}) = \begin{pmatrix} e^{i\theta} & 0 \\ 0 & e^{-2i\theta} \end{pmatrix}\!.
\end{equation}
Recall also that $X$ has a unique singular orbit $X_{\mathrm{sing}} = G/H = \CP^2$ given by the zero section, and the principal orbits are equivariantly diffeomorphic to the Aloff--Wallach space $G/K = N_{1,-2}$. \\
\indent In $\S$\ref{sub:InvS1Bundle}, we classify invariant $S^1$-bundles $P_n \to N_{1,-2}$ and $Q_k \to \CP^2$.  Via pullback, each homogeneous $S^1$-bundle $P_n \to N_{1,-2}$ yields an $S^1$-bundle $E_n \to X^*$ over the principal locus.  In Proposition \ref{prop:S1Extension}, we show that every $E_n \to X^*$ admits a unique extension to an \( S^1\)-bundle \( \widetilde{E}_k \rightarrow X\) with \( k = -n\) whose restriction to the zero section is itself a homogeneous bundle over \( \mathbb{CP}^2\).  In $\S$\ref{sub:InvariantConnectionsS1}, we describe the $\SU(3)$-invariant connections on $E_n \to X^*$ via Wang's theorem. \\
\indent In $\S$\ref{sub:InvarInstS1}, we construct explicit examples of instantons in the $S^1$-bundles described.  In Theorem \ref{thm:Spin7Class}, we classify the $\SU(3)$-invariant $L$-$\Spin(7)$-instantons on $E_n \to X^*$, finding a $3$-parameter family modulo gauge for each $L \in \{I,J,K\}$.  In Theorems \ref{thm:SU4Class} and \ref{cor:Sp2}, we observe that each $3$-parameter family contains a $1$-parameter subfamily of $L$-$\SU(4)$-instantons, and that every pair of these $1$-parameter subfamilies intersects at the same point: the unique invariant $\Sp(2)$-instanton on $E_n$. \\
\indent Finally, we ask which of these instantons (if any) extend across the singular locus.  In Theorem \ref{thm:ExtendInstS1}, we prove that for each $L \in \{I,J,K\}$, the collection of invariant $L$-$\Spin(7)$-instantons that extends to all of $X$ forms a one-parameter family modulo gauge.  Moreover, the only $L$-$\SU(4)$-instanton in such a family is the invariant $\Sp(2)$-instanton.

\subsection{Invariant $S^1$-bundles} \label{sub:InvS1Bundle}

\indent \indent We begin by classifying the homogeneous $S^1$-bundles over $G/K = N_{1,-2}$, as well as the homogeneous $S^1$-bundles over $G/H = \CP^2$.

\begin{prop}[Homogeneous $S^1$-bundles \cite{ball2019gauge}] ${}$
\begin{enumerate}[(a)]
\item Up to conjugacy, every Lie group homomorphism $\lambda \colon \U(1) \to S^1$ is of the form $\lambda = \lambda_n$, where
$$\lambda_n(z) = z^n, \ \ n \in \Z.$$
Consequently, every homogeneous $S^1$-bundle over $N_{1,-2}$ is of the form
$$P_n := \frac{\SU(3) \times S^1}{\U(1)} = \frac{\SU(3) \times S^1}{\sim_n}, \ \text{ where } (g,s) \sim_n (gz, s \lambda_n(z)),$$
for $(g,s) \in \SU(3) \times S^1$ and $z \in \U(1)$.
\item Up to conjugacy, every Lie group homomorphism $\chi \colon \U(2) \to S^1$ is of the form $\chi = \chi_k$, where
$$\chi_k(A) = \det(A)^k, \ \ k \in \Z.$$
Consequently, every homogeneous $S^1$-bundle over $\CP^2$ is of the form
$$Q_k := \frac{\SU(3) \times S^1}{\U(2)} = \frac{\SU(3) \times S^1}{\sim_k}, \ \text{ where } (g,s) \sim_k (gA, s \chi_k(A)),$$
for $(g,s) \in \SU(3) \times S^1$ and $A \in \U(2)$.
\end{enumerate}
\end{prop}

\indent Let $\mathrm{pr} \colon X^* \approx (0,\infty) \times N_{1,-2} \to N_{1,-2}$ denote the projection.  For each $n \in \mathbb{Z}$, we have an $S^1$-bundle $E_n := \mathrm{pr}^*P_n \to X^*$ over the principal locus.  On the other hand, letting $\Pi \colon T^*\CP^2 \to \CP^2$ denote the standard projection, for each $k \in \mathbb{Z}$, we have an $S^1$-bundle $\widetilde{E}_k := \Pi^*Q_k \to X$ over the entire manifold. \\
\indent We now ask which pairs $(n,k) \in \Z \times \Z$ have the property that $\widetilde{E}_k \to X$ is an extension of $E_n \to X^*$ across the singular orbit.  The answer is given by:

\begin{prop}[$S^1$-bundle extension] \label{prop:S1Extension} The $S^1$-bundle $\widetilde{E}_k \to X$ is an extension of $E_n \to X^*$ if and only if $n = -k$.
\end{prop}

\begin{proof} For $e^{i\theta} \in \U(1)$, we have $\lambda_n(e^{i\theta}) = e^{ni\theta}$ and $(\chi_k \circ \iota)(e^{i\theta}) = \chi_k( \mathrm{diag}(e^{i\theta}, e^{-2i\theta} ) ) = e^{-ki\theta}$.  The result now follows from the bundle extension lemma \ref{prop:BundleExtension}.
\end{proof}

\subsection{Invariant connections} \label{sub:InvariantConnectionsS1}

\subsubsection{Invariant connections on $P_n$} \label{subsub:InvarConnPnS1}

\indent \indent We now consider $\SU(3)$-invariant connections on the homogeneous $S^1$-bundles $\pi_n \colon P_n \to N_{1,-2}$ associated to $\lambda_n \colon \U(1) \to S^1$, $\lambda_n(z) = z^n$.  For this, let $\mathfrak{m} \subset \mathfrak{su}(3)$ be the reductive complement such that the splitting $\mathfrak{su}(3) = \mathfrak{m} \oplus \mathfrak{u}(1)$ is orthogonal with respect to the Killing form, and let $\mathfrak{s}$ denote the Lie algebra of the structure group $S = S^1$.  Following the discussion in $\S$\ref{sub:InvariantP}, we regard $\SU(3)$-invariant connections on $P_n$ as elements of $\Omega^1(\SU(3); \mathfrak{s})$ obtained as left-invariant extensions of linear maps $\Lambda + d\lambda_n|_e \colon \mathfrak{m} \oplus \mathfrak{u}(1) \to \mathfrak{s}$, where $\Lambda \colon \mathfrak{m} \to \mathfrak{s}$ is $\U(1)$-equivariant. \\
\indent As such, we seek to express such linear maps with respect to the basis $\{H, X_1, \ldots, X_7\}$ of $\mathfrak{su}(3)$ given by (\ref{eq:HBasis})-(\ref{eq:XBasis}) and the standard basis $\{T\}$ of $\mathfrak{s}$.  First, a quick computation shows that $d\lambda_n|_e \colon \mathfrak{u}(1) \to \mathfrak{s}$ is the linear map satisfying $d\lambda_n|_e(H) = nT$.   Second, recalling the $\U(1)$-irreducible decomposition $$\mathfrak{m} = (X_4) \oplus (X_3) \oplus (X_7) \oplus (X_1, X_5) \oplus (X_2, X_6),$$
an application of Schur's Lemma implies that every $\U(1)$-equivariant linear map $\Lambda \colon \mathfrak{m} \to \mathfrak{s}$ is of the form $\Lambda = \Lambda_p$, where
\begin{equation}
\Lambda_p(x_4, x_3, x_7, x_1 + ix_5, x_2 + ix_6) = p_0x_4 - p_1x_3 + p_2x_7,
\end{equation}
for some $p = (p_0, p_1, p_2) \in \R^3$.   Consequently, the $\SU(3)$-invariant connections on $P_n$ may be identified with
$$\alpha_{n,p} := (p_0\kappa + p_1 \nu_1 + p_2\nu_2 + n\zeta) \otimes T \in \Omega^1(\SU(3); \mathfrak{s}).$$
Using (\ref{eq:CurvatureGeneral}), a calculation shows that the curvature of $\alpha_{n,p}$, regarded as an element of $\Omega^2(\SU(3); \mathfrak{s})$, is given by
\begin{align} \label{eq:FalphaS1}
F_{\alpha_{n,p}} & = [ p_0\left( -\mu_1 \wedge \mu_2 - \sigma_1 \wedge \sigma_2 - 2\nu_1 \wedge \nu_2  \right) + p_1\left( -\sigma_1 \wedge \mu_1 + \sigma_2 \wedge \mu_2 + 2\kappa \wedge \nu_2 \right)  \\ \notag
& \ \  + p_2\left( -\sigma_2 \wedge \mu_1 - \sigma_1 \wedge \mu_2 - 2\kappa \wedge \nu_1 \right) + \left( n \mu_1 \wedge \mu_2 - n \sigma_1 \wedge \sigma_2 \right)] \otimes T. 
\end{align}

\subsubsection{Invariant connections on $E_n$}

\indent \indent Next, we consider $\SU(3)$-invariant connections $A$ on $E_n = \mathrm{pr}^*P_n \to X^*$, recalling that $E_n$ is equivariantly isomorphic to the product bundle $(0,\infty) \times P_n \to (0,\infty) \times N_{1,-2}$.  From the discussion in $\S$\ref{sub:InvariantE}, after a gauge transformation, invariant connections on $E_n$ can be written in the form $A = \alpha(t)$ for some family of invariant connections $\alpha(t)$ on $P_n$ parametrized by $t \in (0, \infty)$.  Therefore, every $\SU(3)$-invariant connection on $E_n$ can be identified with an $\mathfrak{s}$-valued $1$-form
\begin{equation*} 
A_{n,p} := (p_0(t)\kappa + p_1(t) \nu_1 + p_2(t)\nu_2 + n\zeta) \otimes T \in \Omega^1((0,\infty) \times \SU(3); \mathfrak{s}),
\end{equation*}
for some functions $p_0, p_1, p_2 \colon (0,\infty) \to \R$.  Conversely, every $1$-form $A_{n,p}$ yields an $\SU(3)$-invariant connection on $E_n \to X^*$.  From equations (\ref{eq:CurvatureFA}) and (\ref{eq:FalphaS1}), the curvature of $A_{n,p}$ is given by
\begin{align} \label{eq:FAS1MC} 
F_{A_{n,p}} & = [dt \wedge (p_0'(t)\kappa + p_1'(t) \nu_1 + p_2'(t)\nu_2)    \\
& \ \  + p_0(t)\left( -\mu_1 \wedge \mu_2 - \sigma_1 \wedge \sigma_2 - 2\nu_1 \wedge \nu_2  \right) + p_1(t)\left( -\sigma_1 \wedge \mu_1 + \sigma_2 \wedge \mu_2 + 2\kappa \wedge \nu_2 \right) \notag \\
& \ \  + p_2(t)\left( -\sigma_2 \wedge \mu_1 - \sigma_1 \wedge \mu_2 - 2\kappa \wedge \nu_1 \right) + \left( n \mu_1 \wedge \mu_2 - n \sigma_1 \wedge \sigma_2 \right)] \otimes T.  \notag
\end{align}
In terms of the $g$-orthonormal coframe $(e^0, \ldots, e^7)$ of (\ref{eq:Coframe}), we have
\begin{align} \label{eq:FAS1}
F_{A_{n,p}} & = \left[ \frac{cp_0'}{2ab}e^{07} + \frac{p_1'}{c} e^{05} + \frac{p_2'}{c}e^{06} - \frac{2p_0}{c^2}e^{56} - \frac{p_0 + n}{a^2}e^{12} + \frac{p_0-n}{b^2} e^{34}   \right. \\
 & \ \  \ \ \ +  \left. \frac{p_1}{ab}\left( e^{13} + e^{24} - e^{67} \right) + \frac{p_2}{ab}\left( e^{23} - e^{14}  + e^{57} \right) \right] \otimes T, \notag
\end{align}
where we have used the fact that $f = 2abc^{-1}$.

\subsection{Invariant instantons} \label{sub:InvarInstS1}

\indent \indent This section is in two parts.  In $\S$\ref{sub:PrincipalLocus}, we establish various classifications of invariant instantons on the $S^1$-bundles $E_n \to X^*$ for all $n \in \Z$.  
Then, in $\S$\ref{subsub:Extension}, we determine precisely which invariant instantons extend across the zero section.

\subsubsection{Instantons on the principal locus} \label{sub:PrincipalLocus}

\indent \indent We begin by classifying the $\SU(3)$-invariant $L$-$\Spin(7)$-instantons on $E_n \to X^*$, finding a $3$-parameter family for each $L \in \{I,J,K\}$ modulo gauge.  In the sequel, we change variables from $t \in (0,\infty)$ to $r \in (1,\infty)$ via (\ref{eq:dtdr}).  Recall also that $A_{n,0} := n\zeta \otimes T \in \Omega^1((0,\infty) \times \SU(3); \mathfrak{s})$ is the canonical connection on $E_n$.  

\begin{thm}[Invariant $\Spin(7)$-instantons on $E_n$] \label{thm:Spin7Class} Let $A$ be an $\SU(3)$-invariant connection on the $S^1$-bundle $E_n \to X^*$ in temporal gauge, so that $A = A_{n,0} + \left( p_0(r)\kappa + p_1(r)\nu_1 + p_2(r)\nu_2 \right) \otimes T$ for some functions $p_0, p_1, p_2 \colon (1,\infty) \to \R$.  Then:
\begin{enumerate}[(a)]
\item $A$ is an $I$-$\Spin(7)$-instanton if and only if
\begin{align*}
p_0(r) & = -\frac{n}{r^2} + \frac{C_0}{r^2(r^4-1)}, & p_1(r) & = C_1  (r^4-1)^{1/2}, & p_2(r) & = C_2 (r^4 - 1)^{-3/2} 
\end{align*}
for some constants $C_0, C_1, C_2 \in \R$.
\item $A$ is a $J$-$\Spin(7)$-instanton if and only if
\begin{align*}
p_0(r) & = - \frac{n}{r^2} + C_0\frac{1-r^4}{r^2}, & p_1(r) & = C_1 (r^4 - 1)^{-3/2}, & p_2(r) & = C_2(r^4 - 1)^{-3/2},
\end{align*}
for some constants $C_0, C_1, C_2 \in \R$.
\item $A$ is a $K$-$\Spin(7)$-instanton if and only if
\begin{align*}
p_0(r) & = -\frac{n}{r^2} + \frac{C_0}{r^2(r^4-1)}, & p_1(r) & = C_1(r^4 - 1)^{-3/2}, & p_2(r) & = C_2(r^4 - 1)^{1/2},
\end{align*}
for some constants $C_0, C_1, C_2 \in \R$.
\end{enumerate}
\end{thm}

\begin{proof} Let $a(r)$, $b(r)$, $c(r)$, $f(r)$ be the functions defined in (\ref{eq:Metric2}), and note that $2ab = cf$.  Throughout the proof, we use the notation $p_j'(t) = \frac{dp_j}{dt}$.  For future use, we observe that
\begin{equation} \label{eq:dpdt}
\frac{dp_j}{dr} = \frac{dt}{dr}  \frac{dp_j}{dt}  =  \frac{r^2}{ \sqrt{r^4-1} }\, p_j'(t).
\end{equation}
\begin{enumerate}[(a)]
\item A calculation using (\ref{eq:FAS1}) and (\ref{eq:CalabiHK})-(\ref{eq:CalabiSpin7}) shows that the $6$-form $\ast F_A + \Phi_I \wedge F_A$ consists of $12$ terms, and so the $I$-$\Spin(7)$-instanton equation yields a system of 12 ODE's for the three functions $p_0, p_1, p_2$.  However, it is easily seen that $9$ of these equations are redundant.  The remaining $3$ equations arise from:
\begin{align*}
0 = (\ast F_A + \Phi_I \wedge F_A) \wedge e^{56} & = \left( -\frac{c}{2ab}p_0' - \left( \frac{1}{a^2} + \frac{1}{b^2} + \frac{2}{c^2} \right)p_0 + n\left( \frac{1}{b^2} - \frac{1}{a^2}\right) \right) \vol_X \otimes T, \\
0 = (\ast F_A + \Phi_I \wedge F_A) \wedge e^{67} & = -\frac{1}{abc}\left( ab p_1' - cp_1 \right) \vol_X \otimes T, \\
0 = (\ast F_A + \Phi_I \wedge F_A) \wedge e^{57} & = \frac{1}{abc}\left( ab p_2' + 3cp_2 \right) \vol_X  \otimes T.
\end{align*}
This gives the system
\begin{align*}
p_0' & = -\frac{2ab}{c}\left( \frac{1}{a^2} + \frac{1}{b^2} + \frac{2}{c^2} \right) p_0 + \frac{2nab}{c}\left( \frac{1}{b^2} - \frac{1}{a^2} \right)\!, & p_1' & = \frac{c}{ab}p_1, & p_2' & = -\frac{3c}{ab}p_2.
\end{align*}
Using (\ref{eq:dpdt}) and substituting the expressions (\ref{eq:Metric2}) for $a(r)$, $b(r)$, $c(r)$, we obtain the ODE system
\begin{align*}
\frac{dp_0}{dr} & = \frac{-6r^4 + 2}{r(r^4-1)}p_0 - \frac{4nr}{r^4-1}, & \frac{dp_1}{dr} & = \frac{2r^3}{r^4-1}p_1, & \frac{dp_2}{dr} & = -\frac{6r^3}{r^4 - 1}p_2.
\end{align*} 
Solving each ODE separately yields the desired functions.
\item A calculation using (\ref{eq:FAS1}) and (\ref{eq:CalabiHK})-(\ref{eq:CalabiSpin7}) shows that the $6$-form $\ast F_A + \Phi_J \wedge F_A$ consists of $12$ terms, and so the $J$-$\Spin(7)$-instanton equation yields a system of 12 ODE's for the three functions $p_0, p_1, p_2$.  However, it is easily seen that $9$ of these equations are redundant.  Repeating the procedure of part (a), the relevant $3$ equations are:
\begin{align*}
p_0' & =  \frac{2ab}{c}\left( \frac{1}{a^2} + \frac{1}{b^2} - \frac{2}{c^2} \right) p_0 + \frac{2nab}{c}\left( \frac{1}{a^2} - \frac{1}{b^2} \right)\!, & p_1' & = -\frac{3c}{ab}p_1, & p_2' & = -\frac{3c}{ab}p_2.
\end{align*}
Using (\ref{eq:dpdt}) and substituting the expressions (\ref{eq:Metric2}) for $a(r)$, $b(r)$, $c(r)$, we obtain the ODE system
\begin{align*}
\frac{dp_0}{dr} & = \frac{2(r^4 + 1)}{r(r^4-1)}p_0 + \frac{4nr}{r^4 - 1}, & \frac{dp_1}{dr} & = -\frac{6r^3}{r^4 - 1}p_1, & \frac{dp_2}{dr} & = -\frac{6r^3}{r^4 - 1}p_2.
\end{align*} 
Solving each ODE separately yields the desired functions.
\item By now the procedure is clear, so we omit the details.  The relevant $3$ equations are:
\begin{align*}
p_0' & = -\frac{2ab}{c}\left( \frac{1}{a^2} + \frac{1}{b^2} + \frac{2}{c^2} \right) p_0 + \frac{2nab}{c}\left( \frac{1}{b^2} - \frac{1}{a^2} \right)\!, & p_1' & =  -\frac{3c}{ab}p_1, & p_2' & = \frac{c}{ab}p_2.
\end{align*}
This yields the ODE system
\begin{align*}
\frac{dp_0}{dr} & = \frac{-6r^4 + 2}{r(r^4-1)}p_0 - \frac{4nr}{r^4-1}, & \frac{dp_1}{dr} & = -\frac{6r^3}{r^4-1}p_1, & \frac{dp_2}{dr} & = \frac{2r^3}{r^4 - 1}p_2,
\end{align*}
which yields the conclusion.
\end{enumerate}
\end{proof}

\indent Note that here, as well as in the following sections, we abuse notation by repeating the use of $C_0, C_1, C_2$ for the 3 parameters across the three cases $L \in \{ I,J,K \}$.

\begin{rmk}[Asymptotic rates] Let $A$ be an $\SU(3)$-invariant connection on $E_n \to X^*$.  One can show that if $A$ is an $I$-$\Spin(7)$-instanton with $C_1 = 0$, or if $A$ is a $J$-$\Spin(7)$-instanton with $C_0 = 0$, or if $A$ is a $K$-$\Spin(7)$-instanton with $C_2 = 0$, then $A$ is asymptotically conical of rate $-3$ with respect to the canonical connection $\alpha_{n,0}$ on $P_n \to N_{1,-2}$.
\end{rmk}

\indent We now classify the invariant $L$-$\SU(4)$-instantons on $E_n \to X^*$, finding a one-parameter family.

\begin{thm}[Invariant $\SU(4)$-instantons on $E_n$] \label{thm:SU4Class} Let $A$ be an $\SU(3)$-invariant connection on the $S^1$-bundle $E_n \to X^*$ in temporal gauge, so that $A = A_{n,0} + \left( p_0(r)\kappa + p_1(r)\nu_1 + p_2(r)\nu_2 \right) \otimes T$ for some functions $p_0, p_1, p_2 \colon (1,\infty) \to \R$.  Then:
\begin{enumerate}[(a)]
\item $A$ is an $I$-$\SU(4)$-instanton if and only if $A$ is an $I$-$\Spin(7)$-instanton with $C_1 = C_2 = 0$, i.e.:
\begin{align*}
p_0(r) & = -\frac{n}{r^2} + \frac{C_0}{r^2(r^4-1)}, & p_1(r) & = 0, & p_2(r) & = 0.
\end{align*}
\item $A$ is a $J$-$\SU(4)$-instanton if and only if $A$ is a $J$-$\Spin(7)$-instanton with $C_0 = C_1 = 0$, i.e.:
\begin{align*}
p_0(r) & = -\frac{n}{r^2}, & p_1(r) & = 0, & p_2(r) & = C_2(r^4 - 1)^{-3/2}.
\end{align*}
\item $A$ is a $K$-$\SU(4)$-instanton if and only if $A$ is a $K$-$\Spin(7)$-instanton with $C_0 = C_2 = 0$, i.e:
\begin{align*}
p_0(r) & = -\frac{n}{r^2}, & p_1(r) & = C_1(r^4 - 1)^{-3/2}, & p_2(r) & = 0.
\end{align*}
\end{enumerate}
\end{thm}

\begin{proof} Part (a) can be deduced by calculating the $6$-form $\ast F_A + \frac{1}{2}\omega_I^2 \wedge F_A$ via (\ref{eq:FAS1}) and (\ref{eq:CalabiHK})-(\ref{eq:CalabiSpin7}).  The equation $\ast F_A + \frac{1}{2}\omega_I^2 \wedge F_A = 0$ yields an ODE system for $p_0, p_1, p_2$ consisting of $12$ equations, of which $4$ are the condition $p_1 = p_2 = 0$, and the remaining $8$ are the $I$-$\Spin(7)$-instanton condition.  Alternatively, part (a) follows by combining Theorems \ref{thm:Spin7Class}(a) and \ref{thm:Spin7Class}(c) with Proposition \ref{prop:SU4Spin7A}.  Proofs of parts (b) and (c) are analogous.
\end{proof}

\indent Recalling that $\Sp(2)$-instantons are precisely those connections that are simultaneously $I$-$\SU(4)$, $J$-$\SU(4)$, and $K$-$\SU(4)$-instantons, we deduce the following:

\begin{thm}[Invariant $\Sp(2)$-instantons on $E_n$] \label{cor:Sp2} Let $A$ be an $\SU(3)$-invariant connection on the $S^1$-bundle $E_n \to X^*$ in temporal gauge, so that $A = A_{n,0} + \left( p_0(r)\kappa + p_1(r)\nu_1 + p_2(r)\nu_2 \right) \otimes T$ for some functions $p_0, p_1, p_2 \colon (1,\infty) \to \R$. Then:
\begin{align*}
A \text{ is an }\Sp(2)\text{-instanton} & \iff A \text{ is an }I\text{-}\Spin(7)\text{-instanton with } (C_0, C_1, C_2) = (0,0,0) \\
& \iff A \text{ is an }J\text{-}\Spin(7)\text{-instanton with } (C_0, C_1, C_2) = (0,0,0) \\
& \iff A \text{ is an }K\text{-}\Spin(7)\text{-instanton with } (C_0, C_1, C_2) = (0,0,0).
\end{align*}
Consequently, the bundle $E_n \to X^*$ admits a unique $\SU(3)$-invariant $\Sp(2)$-instanton up to gauge transformations, namely
$$A = A_{n,0} - \frac{n}{r^2}\kappa \otimes T.$$
Its curvature is given by
$$
F_{A}= n\left[\dfrac{2}{r^3} dr \wedge \kappa + \left( 1+\dfrac{1}{r^2} \right) \mu_1 \wedge \mu_2 -  \left( 1-\dfrac{1}{r^2} \right) \sigma_1 \wedge \sigma_2 + \dfrac{2}{r^2} \nu_1 \wedge \nu_2 \right] \otimes T.
$$
In particular, for $n \neq 0$, the $\Sp(2)$-instanton $A$ is not flat.
\end{thm}

\begin{proof} This can be deduced by solving the ODE system $\ast F_A + \Theta \wedge F_A = 0$, where $\Theta = \frac{1}{3!}(\omega_I^2 + \omega_J^2 + \omega_K^2)$.  Alternatively, the result follows from Corollary \ref{cor:Sp2Spin7relation} together with either Theorem \ref{thm:Spin7Class} or Theorem \ref{thm:SU4Class}.  The curvature of $A$ follows from (\ref{eq:FAS1MC}).
\end{proof}

\noindent Table \ref{table:instantons} summarizes the results from this section.

\begin{table}
\begin{center}
\begin{tabular}{ |l|l|l|l| } 
 \hline
& $\Spin(7)\text{-instantons}$ & $\SU(4)\text{-instantons}$ & $\Sp(2)\text{-instantons}$ \\  \hline
$I-$
& $p_0 = -\dfrac{n}{r^2} +\dfrac{C_0\strut}{r^2(r^4-1)\strut}$
& $p_0 = -\dfrac{n}{r^2} +\dfrac{C_0\strut}{r^2(r^4-1)\strut}$
& $p_0 =-\dfrac{n}{r^2}$ \\
& $p_1 = C_1 (r^4-1)^{1/2}$
& $p_1= 0$
& $p_1= 0$ \\
& $p_2 = C_2 (r^4-1)^{-3/2}$ 
& $p_2= 0$
& $p_2= 0$ \\ \hline
$J-$
& $p_0 = - \dfrac{n}{r^2\strut} + C_0 \dfrac{1-r^4 \strut}{r^2\strut}$  & $p_0 = -\dfrac{n}{r^2}$ & $p_0 = -\dfrac{n}{r^2}$ \\ 
& $p_1 = C_1 (r^4-1)^{-3/2}$ & $p_1 = 0$ & $p_1 =0$ \\ 
& $p_2 = C_2 (r^4-1)^{-3/2}$ & $p_2 = C_2 (r^4-1)^{-3/2}$ & $p_2=0$ \\ \hline
$K-$
& $p_0 = -\dfrac{n}{r^2} +\dfrac{C_0\strut}{r^2(r^4-1)\strut}$ & $p_0 =-\dfrac{n}{r^2}$ & $p_0 = -\dfrac{n}{r^2}$ \\ 
& $p_1 = C_1 (r^4-1)^{-3/2}$ & $p_1 =C_1 (r^4-1)^{-3/2}$ & $p_1 =0$ \\ 
& $p_2 = C_2 (r^4-1)^{1/2}$ & $p_2 = 0$ & $p_2=0$ \\ \hline
\end{tabular}
\caption{Families of $\SU(3)$-invariant $I,J,K$-instantons on $E_n \rightarrow T^*\CP^2 - \CP^2$}
\label{table:instantons}
\end{center}
\end{table}

\subsubsection{Extension across the zero section} \label{subsub:Extension}

\indent \indent We now consider the question of which invariant $L$-$\Spin(7)$-instantons on $E_n \to X^*$ extend across the zero section.  As in Clarke--Oliveira's study of $\Spin(7)$-instantons on asymptotically conical $\Spin(7)$-orbifolds \cite{clarke2021spin}, we observe that the zero section $\CP^2 \subset T^*\CP^2$ has codimension four, allowing us to appeal to the following special case of the Tao--Tian removable singularity theorem \cite[Theorem 1.1]{tao2004singularity}.

\begin{prop}[$\Spin(7)$-instanton extension \cite{tao2004singularity}, \cite{clarke2021spin}] \label{prop:ExtendConnection} Let $A$ be an $L$-$\Spin(7)$-instanton on $E_n \to X^*$.  Then $A$ extends smoothly over the zero section up to gauge if and only if its curvature norm $|F_A|$ remains uniformly bounded near the zero section.
\end{prop}

The proof of Proposition \ref{prop:ExtendConnection} follows exactly as in \cite[Proposition 5]{clarke2021spin}.

\begin{rmk}\label{rmk:bundleequivariance}
We observe that the bundle carrying the smooth extension of the connection $A$ on $E_{\lambda} \rightarrow X^*$ from Proposition \ref{prop:ExtendConnection} is one of the $G$-equivariant extensions $\widetilde{E}_\chi$ of Proposition \ref{prop:BundleExtension}.

Let $\widetilde{E} \rightarrow X$ denote the extension bundle.
In order to identify 
\(\widetilde{E}\) with some \(\widetilde{E}_\chi\), observe that the \(G\)-action on \(E_\lambda\) extends to \(\widetilde{E}\). More precisely, let \(A\) be a \(G\)-invariant connection on \(E_\lambda\) that extends, via a bundle isomorphism \(\Phi\colon E_\lambda \to \widetilde{E}|_{X^*}\), to a connection \(\widetilde{A}\) on \(\widetilde{E} \to X\). Then the action \(\widetilde{L}_g := \Phi \circ L_g \circ \Phi^{-1}\) on \(\widetilde{E}|_{X^*}\) extends uniquely to a smooth \(G\)-action on \(\widetilde{E}\) by bundle automorphisms, \(\widetilde{A}\) is \(G\)-invariant, and \(\Phi\) is \(G\)-equivariant. Indeed, writing \(P_c\) for \(\widetilde{A}\)-parallel transport along a path \(c\), the invariance \(\widetilde{L}_g^*\widetilde{A} = \widetilde{A}\) gives \(\widetilde{L}_g \circ P_c = P_{gc} \circ \widetilde{L}_g\) for paths \(c\) in \(X^*\), so \(\widetilde{L}_g\) extends to the fiber \(\widetilde{E}_y\) over \(y \in X_{\mathrm{sing}}\) as \(P_{gc} \circ \widetilde{L}_g \circ P_c^{-1}\), where \(c\) is a path from a point of \(X^*\) to \(y\). 

Consequently, \(\widetilde{A}\)-parallel transport along the normal geodesics emanating from \(X_{\mathrm{sing}}\) gives a \(G\)-equivariant isomorphism \(\widetilde{E} \cong \Pi^*(\widetilde{E}|_{X_{\mathrm{sing}}})\). Since \(\widetilde{E}|_{X_{\mathrm{sing}}} \to X_{\mathrm{sing}} \cong G/H\) is a homogeneous \(S\)-bundle, Lemma \ref{lem:Construction} gives \(\widetilde{E}|_{X_{\mathrm{sing}}} \cong Q_\chi\) for some homomorphism \(\chi\colon H \to S\), so \(\widetilde{E} \cong \Pi^*Q_\chi = \widetilde{E}_\chi\), with \(\chi \circ \iota\) conjugate to \(\lambda\) by Proposition \ref{prop:BundleExtension}. 
\end{rmk}

\begin{thm}[Extendibility of invariant $\Spin(7)$-instantons] \label{thm:ExtendInstS1} Let $A$ be an $\SU(3)$-invariant connection on $E_n \to X^*$ in temporal gauge, and let $\widetilde{E}_{k} \rightarrow X$ with $k = -n$ be the unique extension.
\begin{enumerate}[(a)]
\item Suppose $A$ is an $I$-$\Spin(7)$-instanton.  Then $A$ extends to a connection on $\widetilde{E}_k \to X$ if and only if $C_0 = C_2 = 0$.
\item Suppose $A$ is a $J$-$\Spin(7)$-instanton.  Then $A$ extends to a connection on $\widetilde{E}_k \to X$ if and only if $C_1 = C_2 = 0$.
\item Suppose $A$ is a $K$-$\Spin(7)$-instanton.  Then $A$ extends to a connection on $\widetilde{E}_k \to X$ if and only if $C_0 = C_1 = 0$.
\end{enumerate}
\end{thm}

\begin{proof} Before beginning, set
$$h(r) = n^2r^{20} - 3n^2r^{16} + 2n(n-C_0)r^{12} + 2(3C_0^2 + 3nC_0 + n^2)r^8 - 3(C_0+n)^2 r^4 + (C_0 + n)^2,$$
and note that $h(1) = 4C_0^2$.  Moreover, when $C_0 = 0$, we have $h(r) =n^2 (r-1)^4 \cdot (r^4+1)(r^2+1)^4(r+1)^4$.
\begin{enumerate}[(a)]
\item Suppose $A$ is an $I$-$\Spin(7)$-instanton.  A calculation gives
\begin{align*}
|F_A|^2 & = 16C_1^2 + \frac{48C_2^2}{(r^4 - 1)^4} + \frac{8}{r^8(r+1)^4(r^2+1)^4} \cdot \frac{ h(r) }{(r-1)^4}.
\end{align*}
This function has a removable singularity at $r = 1$ if and only if $C_0 = C_2 = 0$.  Proposition \ref{prop:ExtendConnection} gives the result.
\item Suppose $A$ is a $J$-$\Spin(7)$-instanton.  A calculation gives
\begin{align*}
|F_A|^2 & = \frac{48(C_1^2 + C_2^2)}{(r^4-1)^4} + \frac{8}{r^8}\left[ (2r^8 + r^4 + 1)C_0^2 - 2n(r^4 + 1) C_0  + n^2(r^4 + 1) \right].
\end{align*}
This function has a removable singularity at $r = 1$ if and only if $C_1 = C_2 = 0$.  Proposition \ref{prop:ExtendConnection} gives the result.
\item Suppose $A$ is a $K$-$\Spin(7)$-instanton.  A calculation gives
\begin{align*}
|F_A|^2 & = \frac{48C_1^2}{(r^4 - 1)^4} + 16C_2^2 + \frac{8}{r^8(r+1)^4(r^2+1)^4} \cdot \frac{ h(r) }{(r-1)^4}.
\end{align*}
This function has a removable singularity at $r = 1$ if and only if $C_0 = C_1 = 0$.  Proposition \ref{prop:ExtendConnection} gives the result.
\end{enumerate}
\end{proof}

\indent In particular, we observe that for each $L \in \{I, J, K\}$, the only invariant $L$-$\SU(4)$-instanton that extends across the zero section is the $\Sp(2)$-instanton.  We state this formally:

\begin{cor}[Extendibility of invariant $\SU(4)$-instantons] Let $A$ be an $\SU(3)$-invariant $L$-$\Spin(7)$-instanton on $\widetilde{E}_k \to X$ in temporal gauge, where $L \in \{I,J,K\}$.  Then $A$ is an $L$-$\SU(4)$-instanton if and only if $A$ is an $\Sp(2)$-instanton.
\end{cor}

\begin{cor} \label{thm:MainS1} On each $S^1$-bundle $\widetilde{E}_k \to T^*\CP^2$ for $k \in \Z$, and for each $L \in \{I,J,K\}$, the collection of $\SU(3)$-invariant $\Phi_L$-$\Spin(7)$-instantons on $\widetilde{E}_k$ is a one-parameter family 
modulo gauge transformations:
\begin{align*}
A_{I,C_1} &= A_{-k,0} + \frac{k}{r^2}\kappa \otimes T + C_1  (r^4-1)^{1/2} \,\nu_1 \otimes T,  \\
A_{J,C_0} &= A_{-k,0} + \left( \frac{k}{r^2 } + C_0\frac{1-r^4}{r^2}\right) \kappa \otimes T, \\
A_{K,C_2} &= A_{-k,0} + \frac{k}{r^2}\kappa \otimes T + C_2(r^4 - 1)^{1/2}\, \nu_2 \otimes T.
\end{align*}
Moreover, the only connection in the above families that is primitive Hermitian Yang--Mills with respect to $L$ is 
$$A_{L,0} = A_{-k,0} + \frac{k}{r^2}\kappa \otimes T,$$
which is precisely the unique invariant $\Sp(2)$-instanton. 
\end{cor}

It has been brought to the authors’ attention that the construction of invariant instantons on $S^1$-bundles over the Calabi manifold $T^*\CP^2$ was also independently studied in \cite{Fowdar24}. However, \cite{Fowdar24} does not address the matter of the extension of the bundles over $X^*=T^*\CP^2 - \CP^2$ to bundles over $T^*\CP^2$. Similarly, it does not show that the instantons on bundles over $X^*$ extend smoothly, after gauge if necessary, to connections on the extension bundle.

\section{Invariant instantons on $\SO(3)$-bundles over $T^*\CP^2$} \label{sec:InvarInstSO3}

\indent \indent We now seek to construct and classify instantons over $X = T^*\CP^2$ with gauge group $S = \SO(3)$.  Our approach in this section largely follows that of the previous one.  That is, in $\S$\ref{sub:InvarSO3bundles}, we classify the invariant $\SO(3)$-bundles $P_\lambda \to N_{1,-2}$ and $Q_\chi \to \CP^2$. We see that the former ones are classified by $P_n$, $n\geq 0$ and the later ones are either $Q_k$, $k\geq0$, or $Q_\wedge$. Letting $E_n \to X^*$ denote the pullback of $P_n$ to the principal locus, we prove in Proposition \ref{prop:SO3Extension} that every $E_n$ admits a suitable extension to $X$.  Moreover, the extension is unique if $n \neq 3$, whereas the bundle $E_3$ admits two distinct extensions, $\widetilde{E}_3 = \Pi^* Q_3$ and $\widetilde{E}_\wedge = \Pi^* Q_\wedge$. Then, in $\S$\ref{sub:InvarConnSO3}, we apply Wang's theorem to describe the $\SU(3)$-invariant connections on $P_n \to N_{1,-2}$, and hence also on $E_n \to X^*$.  \\
\indent In $\S$\ref{sub:InvarInstantonSO3}, we construct and classify instantons.  
There are then three cases to consider: the bundles $E_{n}$ with $n \neq 0, 3$, the trivial bundle $E_0$, and the bundle $E_3$.  In the first case, the invariant $L$-$\Spin(7)$-instantons essentially reduce to those on the corresponding $S^1$-bundles.  In particular, we find a $3$-parameter family of $L$-$\Spin(7)$-instantons on $E_{n} \to X^*$, of which a $1$-parameter subfamily consists of $L$-$\SU(4)$-instantons.  Moreover, the collection of $L$-$\Spin(7)$-instantons that extends to all of $X$ forms a $1$-parameter family modulo gauge, and the only $L$-$\SU(4)$-instanton in that family is the unique $\Sp(2)$-instanton: see Theorem \ref{thm:InvarInstE2k} and Corollary \ref{cor:InvarInstE2k}. \\
\indent The situation is rather different on the trivial $\SO(3)$-bundle $E_0 \to X^*$.  In Theorem \ref{thm:InstantonE0}, we find that the collection of $L$-$\SU(4)$-instantons on $E_0 \to X^*$ forms a $3$-parameter family.  However, in Theorem \ref{thm:NoExtensionSO3}, we show that the only one that extends over the singular orbit is the (flat) canonical connection.  We do not attempt a classification of $L$-$\Spin(7)$-instantons on $E_0 \to X^*$, leaving this to the interested reader. \\
\indent The case $n=3$ is also distinguished because, unlike in the two previous situations, it yields an $\SU(3)$-invariant $\Sp(2)$-instanton on $X$ that does not reduce to a $S^1$-connection. On each extension, we find exactly one $\Sp(2)$-instanton up to gauge transformation: a reducible one on $\widetilde{E}_3$ and an irreducible one on $\widetilde{E}_\wedge$. We only classify $\Sp(2)$-instantons on $E_3 \to X^*$.

\subsection{Invariant $\SO(3)$-bundles} \label{sub:InvarSO3bundles}

\indent \indent As in the previous section, the first step is to classify the homogeneous $\SO(3)$-bundles over $N_{1,-2}$ and $\CP^2$.  For this, we need the following classification of homomorphisms $\lambda \colon \U(1) \to \SO(3)$ and $\chi \colon \U(2) \to \SO(3)$ up to element conjugacy.

\begin{prop}[Homomorphisms into $\SO(3)$] \label{prop:HomSO3} ${}$
\begin{enumerate}[(a)]
\item Up to conjugacy, every Lie group homomorphism $\lambda \colon \U(1) \to \SO(3)$ is of the form $\lambda = \lambda_n$, where
$$\lambda_n(e^{i\theta}) = \begin{pmatrix}
1 & 0 & 0 \\
0 & \cos(n\theta) & -\sin(n\theta) \\
0 & \sin(n\theta) & \cos(n\theta)
\end{pmatrix}, \ \ n \in \Z_{\geq 0}.$$
\item Up to conjugacy, every Lie group homomorphism $\chi \colon \U(2) \to \SO(3)$ is of the form $\chi = \chi_k$ for $k \in \Z_{\geq 0}$ or $\chi = \widehat{\chi}$, where
\begin{align*}
\chi_k \colon \U(2)  &\to \SO(3); \\
 A &\mapsto \chi_k(A) = \begin{pmatrix}
1 & 0 \\
0 & \det(A)^k
\end{pmatrix},
\end{align*}
and
\begin{align*}
\widehat{\chi} \colon \U(2) \cong \dfrac{\SU(2) \times S^1}{\Z_2} &\to \SO(3); \\
[B,z] &\mapsto \widehat{\chi}( [B,z]) = \varpi(B),  
\end{align*}
where $\varpi \colon \SU(2) \to \SO(3)$ is the standard double cover, and where we are viewing $\det(A)^k \in \U(1) \cong \SO(2)$ as a real $(2 \times 2)$ matrix in the standard way.
\end{enumerate}
\end{prop}

The proof of the previous Proposition follows from the following remark.
\begin{rmk}\label{rmk:conjugates}
There is a correspondence between group homomorphisms $\lambda: H \rightarrow \SO(3)$ up to conjugation and $H$-representations of $\R^3$ up to isomorphism, where $H$ is a compact, connected Lie group.
To prove this, we consider a representation $\rho: H \rightarrow \GL(3, \R)$. 
Let $\nu$ be the Haar measure of $H$, we define an inner product 
$
\langle x, y \rangle_H = \int_H \langle \rho(h)x, \rho(h)y \rangle d \nu(h), 
$
which by construction is invariant under the action of $H$ given by $\rho$. 
From the Gram--Schmidt process, we know there exists an invertible map $f: \R^3 \rightarrow \R^3$ such that 
$
\langle x,y\rangle = \langle fx,fy\rangle_H.
$
Let $\lambda(h):= f^{-1} \rho(h) f$, where we are abusing notation and now using $f$ for the associated matrix, then $\langle \lambda(h) x, \lambda(h) y \rangle=\langle x, y \rangle$ follows from a quick computation, so $\lambda (h) \in \mathrm{O}(3)$. Moreover, $\rho$ being a representation implies that $\lambda: H \rightarrow \mathrm{O}(3)$ is a group homomorphism.
Finally, since $H$ is connected, its image under the group homomorphism $\lambda$ has to lie in the connected component of $e \in \mathrm{O}(3)$, which is indeed $\SO(3)$.
Conversely, if $\lambda: H \rightarrow \SO(3)$ is a group homomorphism, we can write a representation $\rho: H \rightarrow \GL(3,\R)$ by $h \mapsto \rho(h) = \lambda(h)$.\\
We observe that conjugate group homomorphisms correspond to isomorphic representations. For the converse, we also need to use the fact that if two $\SO(3)$ matrices are conjugate in $\GL(3,\R)$, they are also conjugate in $\SO(3)$, which follows from the fact that $\SO(3)$ matrices that are conjugate have the same angle of rotation.
\end{rmk}
\begin{proof}
The only irreducible, real $S^1$-representations are the trivial $ \mathbb{R}$ and the 2-dimensional $\LB A_n\RB$ with $n \geq 1$, where $A_n = \mathbb{C}$ is the irreducible, complex $S^1$-representation given by $e^{i\theta} \cdot z := e^{in\theta} z$ and we are using the Salamon bracket notation. Since $A_{-n}$ is the conjugate of $A_n$, we have $ \LB A_{n}\RB \cong \LB A_m \RB$ if and only if $n = \pm m$. Since every odd-dimensional real representation of $S^1$ contains a trivial summand, every representation of $S^1$ on $\mathbb{R}^{3}$ is isomorphic to $\mathbb{R} \oplus \LB A_n \RB$ for a unique $n \geq 0$. Therefore, any homomorphism $\lambda: \U(1) \rightarrow \SO(3)$ is element conjugate to $\lambda_n$ for a unique $n \geq 0$, as in $(a)$. 

The three-dimensional real $\U(2)$-representations, after
pullback to $\SU(2)\times S^1$ along $(B,z)\mapsto zB$, are $\R\oplus\LB\mathsf A_{2k}\RB$, $k\geq0$ and $[\mathsf U_2]=[\Sym^2(\C^2)]$.
The restriction to $\SU(2)$ is either trivial or its irreducible 3-dimensional real representation; in the latter case, the commuting $S^1$-factor must act trivially. The even weight $2k$ in the former case is precisely the condition that $(-I,-1)$ act trivially, so that the representation descends to $\U(2)$.
These representations correspond respectively to $\chi_k$ and $\widehat{\chi}$ in (b).
\end{proof}

\begin{cor}[Homogeneous $\SO(3)$-bundles] ${}$
\begin{enumerate}[(a)]
\item Every homogeneous $\SO(3)$-bundle over $N_{1,-2}$ is of the form
$$P_n := \frac{\SU(3) \times \SO(3)}{\U(1)} = \frac{\SU(3) \times \SO(3)}{\sim_n}, \ \text{ where } (g,s) \sim_n (gz, s \lambda_n(z)),$$
for $(g,s) \in \SU(3) \times \SO(3)$, $z \in \U(1)$, and $n \in \Z_{\geq 0}$.
\item Every homogeneous $\SO(3)$-bundle over $\CP^2$ is of one of the following two forms:
\begin{align*}
Q_k & := \frac{\SU(3) \times \SO(3)}{\U(2)} = \frac{\SU(3) \times \SO(3)}{\sim_k}, \ \text{ where } (g,s) \sim_k (gA, s\chi_k(A)), \\
Q_\wedge & := \frac{\SU(3) \times \SO(3)}{\U(2)} = \frac{\SU(3) \times \SO(3)}{\sim_\wedge}, \ \text{ where } (g,s) \sim_\wedge (gA, s \widehat{\chi}(A)),
\end{align*}
for $(g,s) \in \SU(3) \times \SO(3)$, $A \in \U(2)$, and $k \in \Z_{\geq 0}$. \\
\end{enumerate}
\end{cor}

\indent Let $\mathrm{pr} \colon X^* \approx (0,\infty) \times N_{1,-2} \to N_{1,-2}$ denote the projection.  For each $n \in \mathbb{Z}_{\geq 0}$, we have an $\SO(3)$-bundle $E_n := \mathrm{pr}^*P_n \to X^*$ over the principal locus.  On the other hand, letting $\Pi \colon T^*\CP^2 \to \CP^2$ denote the standard projection, we have $\SO(3)$-bundles $\widetilde{E}_k := \Pi^*Q_k \to X$ and $\widetilde{E}_\wedge := \Pi^*Q_\wedge \to X$ over the entire manifold.  We now address the matching question:

\begin{prop}[$\SO(3)$-bundle extension] \label{prop:SO3Extension} Let \( n \in \mathbb{Z}_{\geq 0}\). 
	\begin{enumerate}[label=(\alph*)]
		\item The bundle \( \widetilde{E}_{k}\) is an extension of \( E_n\) if and only if \( n = k\). 
		\item The bundle \( \widetilde{E}_{\wedge}\) is an extension of \( E_n\) if and only if \( n = 3\).
\end{enumerate}\end{prop}

\begin{proof} We use the bundle extension lemma \ref{prop:BundleExtension}, and that for any \( n, m \in \mathbb{Z}\), the homomorphisms \( \lambda_n\) and \( \lambda_m\) are element conjugate if and only if \( m = \pm n\). We compute
	\begin{equation*}
		\left( \chi_k \circ \iota \right)\left( e^{i\theta} \right)
        = \chi_k \begin{pmatrix} e^{i \theta} & 0  \\ 0 & e^{-2i\theta} \end{pmatrix}
        = \begin{pmatrix} 1 & 0  \\ 0 & e^{-ik\theta} \end{pmatrix}= \lambda_{-k}\left( e^{i\theta} \right).
	\end{equation*}
	and
    $$(\widehat{\chi} \circ \iota)(e^{i\theta}) = \widehat{\chi}\begin{pmatrix} e^{i\theta} & 0 \\ 0 & e^{-2i\theta} \end{pmatrix} = \widehat{\chi}\left(\left[ 
    \begin{pmatrix} e^{3i\theta / 2} & 0 \\ 0 & e^{-3 i \theta / 2} \end{pmatrix}
    , e^{-i\theta /2} \right]\right) 
    = \varpi \begin{pmatrix} e^{3i\theta / 2} & 0 \\ 0 & e^{-3 i \theta / 2} \end{pmatrix}
    = \lambda_3(e^{i\theta}).$$
\end{proof}

\subsection{Invariant connections} \label{sub:InvarConnSO3}

\indent \indent We now consider $\SU(3)$-invariant connections on the homogeneous $\SO(3)$-bundles $\pi_n \colon P_n \to N_{1,-2}$ associated to the maps $\lambda_n \colon \U(1) \to \SO(3)$ of Proposition \ref{prop:HomSO3}. As in previous sections, we choose $\mathfrak{m} \subset \mathfrak{su}(3)$ such that the splitting $\mathfrak{su}(3) = \mathfrak{m} \oplus \mathfrak{u}(1)$ is orthogonal with respect to the Killing form. Also as in previous sections, we regard $\SU(3)$-invariant connections on $P_n$ as elements of $\Omega^1(\SU(3); \mathfrak{so}(3))$ obtained as left-invariant extensions of linear maps $\Lambda + d\lambda_n|_e \colon \mathfrak{m} \oplus \mathfrak{u}(1) \to \mathfrak{so}(3)$, where $\Lambda \colon \mathfrak{m} \to \mathfrak{so}(3)$ is $\U(1)$-equivariant. \\
\indent To be clear, the $\U(1)$-action on $\mathfrak{so}(3)$ is the $(\mathrm{Ad} \circ \lambda)$-action, whereby $z \in \U(1)$ acts on $T \in \mathfrak{so}(3)$ via $z \cdot T := \lambda(z) T \lambda(z)^{-1}$. In terms of the following basis $\{T_1, T_2, T_3\}$ of $\mathfrak{so}(3)$
\begin{align*}
T_1 & = \begin{bmatrix} 0 & 0 & 0 \\ 0 & 0 & -1 \\ 0 & 1 & 0 \end{bmatrix}\!, &
T_2 & = \begin{bmatrix} 0 & 0 & 1 \\ 0 & 0 & 0 \\ -1 & 0 & 0 \end{bmatrix}\!, &
T_3 & = \begin{bmatrix} 0 & -1 & 0 \\ 1 & 0 & 0 \\ 0 & 0 & 0 \end{bmatrix}\!,
\end{align*}
we see that $e^{i\theta} \cdot T_1 = T_1$ and $e^{i\theta} \cdot T_2 = \cos(n\theta)T_2 + \sin(n\theta)T_3$ and $e^{i\theta} \cdot T_3 = -\sin(n\theta)T_2 + \cos(n\theta)T_3$. As such, we have a $\U(1)$-irreducible decomposition
\begin{align*}
\mathfrak{so}(3) & \cong (T_1) \oplus \C_n, \text{ if } n \neq 0, \\
\mathfrak{so}(3) & \cong (T_1) \oplus (T_2) \oplus (T_3), \text{ if } n = 0,
\end{align*}
where $\C_n$ denotes the real $\U(1)$-representation on $\C$ via $e^{i\theta} \cdot z = e^{ni\theta}z$ for $n \in \Z_{\geq 0}$. \\
\indent Now, let $\{H, X_1, \ldots, X_7\}$ be the basis of $\mathfrak{su}(3)$ given by (\ref{eq:HBasis})-(\ref{eq:XBasis}). Recalling the irreducible decomposition of real $\U(1)$-modules given by
\begin{align*}
\mathfrak{m} & = (X_4) \oplus (X_3) \oplus (X_7) \oplus (X_1, X_5) \oplus (X_2, X_6) \\
& \cong \R \oplus \R \oplus \R \oplus \C_3 \oplus \C_3,
\end{align*}
an application of Schur's lemma implies that every $\U(1)$-equivariant linear map $\Lambda \colon \mathfrak{m} \to \mathfrak{so}(3)$ takes one of the following three forms, where $x = (x_4, x_3, x_7, (x_1,x_5), (x_2, x_6)) \in \mathfrak{m}$:
\begin{enumerate}
\item For $n \neq 0, 3$:
\begin{align*}
\Lambda_p(x) & = \left(p_0 x_4 - p_1x_3 + p_2 x_7\right) T_1,
\end{align*}
for $p = (p_0, p_1, p_2) \in \R^3$.
\item For $n = 0$: 
\begin{align*}
\Lambda_{p,q,s}(x) & = \left(p_0 x_4 - p_1x_3 + p_2 x_7\right) T_1 + \left(q_0 x_4 - q_1x_3 + q_2 x_7\right) T_2 + \left(s_0 x_4 - s_1x_3 + s_2 x_7\right) T_3,
\end{align*}
for $p, q, s \in \R^3$.
\item For $n = 3$:
\begin{align*}
	\Lambda_{p,w,\tilde{w}}(x) & = \left(p_0 x_4 - p_1x_3 + p_2 x_7\right) T_1 + \left( w_1 x_1 - \tilde{w}_1 x_5 - w_2 x_2 + \tilde{w}_2 x_6 \right)T_2 \\
	& \ \ \ \ \ \ + \left( \tilde{w}_1 x_1 + w_1 x_5 - \tilde{w}_2 x_2 - w_2 x_6 \right)T_3,
\end{align*}
for $p = (p_0, p_1, p_2) \in \R^3$, $w = (w_1, w_2) \in \R^2$ and $\tilde{w}=(\tilde{w}_1, \tilde{w}_2) \in \R^2$.
\end{enumerate}
On the other hand, a short calculation shows that $d\lambda_n|_e \colon \mathfrak{u}(1) \to \mathfrak{so}(3)$ is the linear map satisfying $d\lambda_n|_e(H) = nT_1$. \\
\indent We now use this information to classify the $\SU(3)$-invariant connections on $E_n = \mathrm{pr}^*P_n \to X^*$, where again we recall that $E_n$ is equivariantly diffeomorphic to the product bundle $(0,\infty) \times P_n \to (0,\infty) \times N_{1,-2}$. Repeating the logic employed in $\S$\ref{sub:InvariantConnectionsS1}, we arrive at the following:

\begin{prop}[Invariant connections on $E_n$] \label{prop:SchurSO3} ${}$
\begin{enumerate}
	\item[(a)] For $n \in \mathbb{Z}_{\geq 0}$ with $n \neq 0,3$, every $\SU(3)$-invariant connection $A$ on $E_n \to X^*$ in temporal gauge can be identified with an $\mathfrak{so}(3)$-valued $1$-form
\begin{align*}
A = \left( p_0(t) \kappa + p_1(t) \nu_1 + p_2(t) \nu_2 + n\zeta \right) \otimes T_1,
\end{align*}
for some functions $p_j \colon (0,\infty) \to \R$ with $j = 0,1,2$. Conversely, every such $1$-form yields an $\SU(3)$-invariant connection on $E_n \to X^*$ in temporal gauge.

\item[(b)] For $n = 0$, every $\SU(3)$-invariant connection $A$ on $E_0 \to X^*$ in temporal gauge can be identified with an $\mathfrak{so}(3)$-valued $1$-form
\begin{align*}
A & = \left( p_0(t) \kappa + p_1(t) \nu_1 + p_2(t) \nu_2 \right) \otimes T_1 + \left( q_0(t) \kappa + q_1(t) \nu_1 + q_2(t) \nu_2 \right) \otimes T_2 \\
& \ \ \ \ \ \ + \left( s_0(t) \kappa + s_1(t) \nu_1 + s_2(t) \nu_2 \right) \otimes T_3,
\end{align*}
for some functions $p_j, q_j, s_j \colon (0,\infty) \to \R$ with $j = 0,1,2$. Conversely, every such $1$-form yields an $\SU(3)$-invariant connection on $E_0 \to X^*$ in temporal gauge.

\item[(c)] For $n = 3$, every $\SU(3)$-invariant connection $A$ on $E_3 \to X^*$ in temporal gauge can be identified with an $\mathfrak{so}(3)$-valued $1$-form
\begin{align*}
A & = \left( p_0(t) \kappa + p_1(t) \nu_1 + p_2(t) \nu_2 + 3\zeta \right) \otimes T_1 \\
& \ \ \ \ \ \ + \left( w_1(t) \sigma_1 - \tilde{w}_1(t) \sigma_2 + w_2(t) \mu_1 + \tilde{w}_2(t) \mu_2 \right) \otimes T_2 \\
& \ \ \ \ \ \ + \left( \tilde{w}_1(t) \sigma_1 + w_1(t) \sigma_2 + \tilde{w}_2(t) \mu_1 - w_2(t) \mu_2 \right) \otimes T_3,
\end{align*}
for some functions $p_0, p_1, p_2, w_1, \tilde{w}_1, w_2, \tilde{w}_2 \colon (0,\infty) \to \R$. Conversely, every such $1$-form yields an $\SU(3)$-invariant connection on $E_3 \to X^*$ in temporal gauge.
\end{enumerate}
\end{prop}

\subsection{Invariant instantons} \label{sub:InvarInstantonSO3}

\indent \indent We now aim to classify invariant instantons on the bundles $E_n \to X^*$. In light of the case splitting of Proposition \ref{prop:SchurSO3}, we will consider the cases of $n \neq 0,3$, \( n = 0\), and $n = 3$ separately. Going forward, we change variables from $t \in (0,\infty)$ to $r \in (1,\infty)$ via (\ref{eq:dtdr}).

\subsubsection{Invariant instantons on $E_{n}$ for $n \neq 0, 3$} \label{subsub:InvarInstE2k}

\indent \indent Let us consider invariant instantons on $E_{n} \to X^*$ for $n \neq 0, 3$. In view of Proposition \ref{prop:SchurSO3}(a), we see that the invariant connections on the $\SO(3)$-bundles $E_{n} \to X^*$ essentially reduce to those on the corresponding $S^1$-bundles (which we also called $E_{n}$). As such, the relevant equations will be exactly the same, and the conclusions of $\S$\ref{sub:InvarInstS1} hold \emph{mutatis mutandis}. Letting $A_{n,0} = n\zeta \otimes T_1$ denote the canonical connection on $E_{n} \to X^*$ for $n \neq 0,3$, we record the result here:

\begin{thm}[Invariant instantons on $E_{n}$ for $n \neq 0,3$] \label{thm:InvarInstE2k} Let $A$ be an $\SU(3)$-invariant connection in temporal gauge on the $\SO(3)$-bundle $E_{n} \to X^*$ for $n \neq 0,3$, so that $A = A_{n,0} + ( p_0(r) \kappa + p_1(r) \nu_1 + p_2(r) \nu_2 ) \otimes T_1$ for some functions $p_0, p_1, p_2 \colon (1, \infty) \to \R$. Then:
\begin{enumerate}[(a)]
\item $A$ is an $I$-$\Spin(7)$-instanton if and only if
\begin{align*}
	p_0(r) & = -\frac{n}{r^2} + \frac{C_0}{r^2(r^4-1)}, & p_1(r) & = C_1  (r^4-1)^{1/2}, & p_2(r) & = C_2 (r^4 - 1)^{-3/2} 
\end{align*}
for some constants $C_0, C_1, C_2 \in \R$. Moreover, in this case:
\begin{itemize}
\item $A$ is an $I$-$\SU(4)$-instanton if and only if $C_1 = C_2 = 0$.
\item $A$ extends to a connection on $\widetilde{E}_{n} \to X$ if and only if $C_0 = C_2 = 0$.
\end{itemize}
\item $A$ is a $J$-$\Spin(7)$-instanton if and only if
\begin{align*}
p_0(r) & = - \frac{n}{r^2} + C_0\frac{1-r^4}{r^2}, & p_1(r) & = C_1 (r^4 - 1)^{-3/2}, & p_2(r) & = C_2(r^4 - 1)^{-3/2},
\end{align*}
for some constants $C_0, C_1, C_2 \in \R$. Moreover, in this case:
\begin{itemize}
\item $A$ is a $J$-$\SU(4)$-instanton if and only if $C_0 = C_1 = 0$.
\item $A$ extends to a connection on $\widetilde{E}_{n} \to X$ if and only if $C_1 = C_2 = 0$.
\end{itemize}
\item $A$ is a $K$-$\Spin(7)$-instanton if and only if
\begin{align*}
p_0(r) & = -\frac{n}{r^2} + \frac{C_0}{r^2(r^4-1)}, & p_1(r) & = C_1(r^4 - 1)^{-3/2}, & p_2(r) & = C_2(r^4 - 1)^{1/2},
\end{align*}
for some constants $C_0, C_1, C_2 \in \R$. Moreover, in this case:
\begin{itemize}
\item $A$ is a $K$-$\SU(4)$-instanton if and only if $C_0 = C_2 = 0$.
\item $A$ extends to a connection on $\widetilde{E}_{n} \to X$ if and only if $C_0 = C_1 = 0$.
\end{itemize}
\end{enumerate}
\end{thm}

\begin{cor}[Invariant $\Sp(2)$-instantons on $E_{n}$ for $n \neq 0, 3$] \label{cor:InvarInstE2k} For each $n \neq 0, 3$, the $\SO(3)$-bundle $E_{n} \to X^*$ admits a unique $\SU(3)$-invariant $\Sp(2)$-instanton up to gauge transformations, namely
$$A = A_{n,0} - \frac{n}{r^2}\kappa \otimes T_1.$$
Moreover, for each $n \neq 0,3$, this $\Sp(2)$-instanton is not flat and extends to a connection on $\widetilde{E}_{n} \to X$.
\end{cor}

\subsubsection{Invariant instantons on $E_0$}

\indent \indent We now turn our attention to instantons on $E_0 \to X^*$.  In view of the complexity of the $L$-$\Spin(7)$-instanton equations in this setting, we content ourselves with a classification of the $L$-$\SU(4)$-instantons.  Going forward, we let $(\cdot, \cdot)$ denote the multiple of the Killing form on $\mathfrak{so}(3)$ with respect to which $\{T_1, T_2, T_3\}$ is orthonormal.

\begin{thm}[Invariant $\SU(4)$-instantons on $E_0$] \label{thm:InstantonE0} Let $A$ be an $\SU(3)$-invariant connection on $E_0 \to X^*$ in temporal gauge, so that
\begin{align*}
A & = \left( p_0(r) \kappa + p_1(r) \nu_1 + p_2(r)\nu_2 \right) \otimes T_1 + \left( q_0(r) \kappa + q_1(r) \nu_1 + q_2(r)\nu_2 \right) \otimes T_2 \\
& \ \ \ \ \ \  + \left( s_0(r) \kappa + s_1(r) \nu_1 + s_2(r)\nu_2 \right) \otimes T_3,
\end{align*}
for some functions $p_j, q_j, s_j \colon (1,\infty) \to \R$ with $j = 0,1,2$.  Then:
\begin{enumerate}[(a)]
\item $A$ is an $I$-$\SU(4)$-instanton if and only if
\begin{align*}
p_0(r) & = \frac{C_1}{r^2(1-r^4)}, & q_0(r) & = \frac{C_2}{r^2(1-r^4)}, & s_0(r) & = \frac{C_3}{r^2(1-r^4)},
\end{align*}
and $p_1 = p_2 = q_1 = q_2 = s_1 = s_2 = 0$, for some constants $C_1, C_2, C_3 \in \R$.
\item $A$ is a $J$-$\SU(4)$-instanton if and only if
\begin{align*}
p_2(r) & = \frac{C_1}{(r^4 - 1)^{3/2}}, & q_2(r) & = \frac{C_2}{(r^4 - 1)^{3/2}}, & s_2(r) & = \frac{C_3}{(r^4 - 1)^{3/2}},
\end{align*}
and $p_0 = p_1 = q_0 = q_1 = s_0 = s_1 = 0$, for some constants $C_1, C_2, C_3 \in \R$.
\item $A$ is a $K$-$\SU(4)$-instanton if and only if
\begin{align*}
p_1(r) & = \frac{C_1}{(r^4 - 1)^{3/2}}, & q_1(r) & = \frac{C_2}{(r^4 - 1)^{3/2}}, & s_1(r) & = \frac{C_3}{(r^4 - 1)^{3/2}},
\end{align*}
and $p_0 = p_2 = q_0 = q_2 = s_0 = s_2 = 0$, for some constants $C_1, C_2, C_3 \in \R$.
\end{enumerate}
\end{thm}

\noindent To be more explicit, the $I$-, $J$-, and $K$-$\SU(4)$-instantons of Theorem \ref{thm:InstantonE0} are, respectively,
\begin{align*}
A_I & = \frac{1}{r^2(1-r^4)} \kappa \otimes T_0, & A_J & = \frac{1}{(r^4-1)^{3/2}} \nu_2 \otimes T_0, & A_K & = \frac{1}{(r^4-1)^{3/2}} \nu_1 \otimes T_0,
\end{align*}
where here $T_0 = C_1T_1 + C_2 T_2 + C_3 T_3 \in \mathfrak{so}(3)$ is arbitrary.  We now give the proof.

\begin{proof} Let $a(r)$, $b(r)$, $c(r)$, $f(r)$ be the functions defined in (\ref{eq:Metric2}), and note that $2ab = cf$.  We use the notation $p_j'(t) = \frac{dp_j}{dt}$ and similarly for $q_j'(t)$ and $s_j'(t)$.  For future use, we observe that
\begin{equation} \label{eq:dpdt2}
\frac{dp_j}{dr} = \frac{dt}{dr}  \frac{dp_j}{dt}  =  \frac{r^2}{ \sqrt{r^4-1} }\, p_j'(t).
\end{equation}
For $L \in \{I,J,K\}$, let us set
$$H_L := \ast F_A + \frac{1}{2}\omega_L^2 \wedge F_A.$$
\begin{enumerate}[(a)]
\item We seek to solve the equation $H_I = 0$.  To that end, we calculate that
\begin{align*}
H_I \wedge e^{24} = \frac{2}{ab}\vol_X (p_1 T_1 +q_1 T_2 +s_1 T_3), \quad
H_I \wedge e^{14} = -\frac{2}{ab}\vol_X (p_2 T_1 + q_2 T_2 + s_2 T_3).
\end{align*}
Thus, if $A$ is an $I$-$\SU(4)$-instanton, then $p_1 = p_2 = q_1 = q_2 = s_1 = s_2 = 0$.  Imposing these necessary conditions, the equation $H_I = 0$ reduces to the following three equations:
\begin{align*}
p_0' & = -\left( \frac{2a}{bc} + \frac{2b}{ac} + \frac{4ab}{c^3} \right)p_0, & q_0' & = -\left( \frac{2a}{bc} + \frac{2b}{ac} + \frac{4ab}{c^3} \right)q_0, & s_0' & = -\left( \frac{2a}{bc} + \frac{2b}{ac} + \frac{4ab}{c^3} \right)s_0.
\end{align*}
Using (\ref{eq:dpdt2}) and substituting the expressions (\ref{eq:Metric2}) for $a(r)$, $b(r)$, $c(r)$, we obtain the following ODE system:
\begin{align*}
\frac{dp_0}{dr} & = -\frac{2(3r^4 - 1)}{r(r^4 - 1)}p_0, & \frac{dq_0}{dr} & = -\frac{2(3r^4 - 1)}{r(r^4 - 1)}q_0, & \frac{ds_0}{dr} & = -\frac{2(3r^4 - 1)}{r(r^4 - 1)}s_0.
\end{align*}
Solving each ODE separately yields the result.
\item We seek to solve the equation $H_J = 0$.  To that end, we calculate that
\begin{align*}
(H_J, T_1) \wedge e^{13} & = \frac{2}{ab}p_1\,\vol_X, & (H_J, T_1) \wedge e^{34} & = \frac{a^2+b^2}{a^2 b^2}p_0\,\vol_X \\
 (H_J, T_2) \wedge e^{13} & = \frac{2}{ab}q_1\,\vol_X, & (H_J, T_2) \wedge e^{12} & = -\frac{a^2+b^2}{a^2 b^2}q_0\,\vol_X \\
(H_J, T_3) \wedge e^{24} & = \frac{2}{ab}s_1\,\vol_X, & (H_J, T_3) \wedge e^{34} & = \frac{a^2+b^2}{a^2 b^2}s_0\,\vol_X.
\end{align*}
Thus, if $A$ is a $J$-$\SU(4)$-instanton, then $p_0 = p_1 = q_0 = q_1 = s_0 = s_1 = 0$.  Imposing these necessary conditions, the equation $H_J = 0$ reduces to the following three equations:
\begin{align*}
p_2' & = -\frac{3c}{ab}p_2, & q_2' & = -\frac{3c}{ab}q_2, & s_2' & = -\frac{3c}{ab}s_2.
\end{align*}
Using (\ref{eq:dpdt2}) and substituting the expressions (\ref{eq:Metric2}) for $a(r)$, $b(r)$, $c(r)$, we obtain the following ODE system:
\begin{align*}
\frac{dp_2}{dr} & = -\frac{6r^3}{r^4 - 1}p_2, & \frac{dq_2}{dr} & = -\frac{6r^3}{r^4 - 1}q_2, & \frac{ds_2}{dr} & = -\frac{6r^3}{r^4 - 1}s_2.
\end{align*}
Solving each ODE separately yields the result.
\item We seek to solve $H_K = 0$.  The calculation is completely analogous to that of part (b).  The relevant ODE system is
\begin{align*}
\frac{dp_1}{dr} & = -\frac{6r^3}{r^4 - 1}p_1, & \frac{dq_1}{dr} & = -\frac{6r^3}{r^4 - 1}q_1, & \frac{ds_1}{dr} & = -\frac{6r^3}{r^4 - 1}s_1,
\end{align*}
whose solutions give the result.
\end{enumerate}
\end{proof}

\begin{cor}[Invariant $\Sp(2)$-instantons on $E_0$] The $\SO(3)$-bundle $E_0 \to X^*$ admits a unique $\SU(3)$-invariant $\Sp(2)$-instanton up to gauge transformation, namely the (flat) canonical connection.
\end{cor}

\indent Finally, we address the extendibility of invariant $L$-$\SU(4)$-instantons on $E_0$ across the zero section.

\begin{thm}[Extendibility of invariant $\SU(4)$-instantons] \label{thm:NoExtensionSO3} Let $A$ be an $\SU(3)$-invariant $L$-$\SU(4)$-instanton on $E_0 \to X^*$ in temporal gauge, where $L \in \{I,J,K\}$.  Then $A$ extends across the zero section up to gauge if and only if $A$ is the (flat) canonical connection.
\end{thm}

\begin{proof} If $A$ is an $I$-$\SU(4)$-instanton, then a calculation gives
$$\left| (F_A, T_j) \right|^2 = \frac{8(6r^8 - 3r^4 + 1)}{r^8 (r^2 + 1)^4 (r+1)^4} \cdot \frac{C_j^2}{(r-1)^4},$$
for $j = 1, 2, 3$.  Each such function is bounded near the zero section if and only if $C_j = 0$.  Thus, by Proposition \ref{prop:ExtendConnection}, the only $I$-$\SU(4)$-instanton that extends across the zero section up to gauge is the one with $C_1 = C_2 = C_3 = 0$, which is the canonical connection.  Similarly, if $A$ is a $J$- or $K$-$\SU(4)$-instanton, then its curvature satisfies
$$\left| (F_A, T_j) \right|^2 = \frac{48C_j^2}{(r^4 - 1)^4},$$
for $j = 1,2,3$.  Again by Proposition \ref{prop:ExtendConnection}, the only $J$- or $K$-$\SU(4)$-instanton that extends across the zero section up to gauge is the one with $C_1 = C_2 = C_3 = 0$.
\end{proof}

\subsubsection{Invariant instantons on \( E_3\)}\label{sub:invariantInstantonsE3}
\indent \indent We now turn to invariant instantons on the bundle \( E_3 \to X^* \). Unlike the bundles \( E_n \) with \( n \neq 0, 3 \), the \( \SU(3) \)-invariant connections on \( E_3 \) need not reduce to those on an \( S^1 \)-bundle.
In this case, we restrict our attention to \( \Sp(2) \)-instantons, since the \( L \)-\( \SU(4) \) and the \( L \)-\( \Spin(7) \)-instanton equations are more involved and we are focusing on explicit solutions. 
We write 
\begin{equation}
	\begin{aligned}
		A = A_{3,0} &+ \left( p_0(r) \kappa + p_1(r)\nu_1 + p_2(r) \nu_2 \right) \otimes T_1 \\
		&+ \left( w_1(r) \sigma_1 - \tilde{w}_1(r) \sigma_2 + w_2(r)\mu_1 + \tilde{w}_2(r)\mu_2 \right) \otimes T_2 \\
		&+ \left( \tilde{w}_1(r) \sigma_1 + w_1(r)\sigma_2 + \tilde{w}_2(r) \mu_1 - w_2(r) \mu_2 \right) \otimes T_3
	\end{aligned}
	\label{eq:A}
\end{equation}
where \( A_{3,0} = 3\zeta \otimes T_1\).

\begin{thm}[Invariant $\Sp(2)$-instantons on $E_3$]
\label{thm:Sp2E3}
The \( \SO(3)\)-bundle \( E_3 \rightarrow X^{*}\) admits exactly two \( \SU(3)\)-invariant \( \Sp(2)\)-instantons up to gauge transformations.
They are
\begin{align}
A_{\mathrm{red}}
&=
A_{3,0}-\frac{3}{r^2}\kappa\otimes T_1,
\label{eq:Sp2E3Red}\\
A_{\wedge}
&=
A_{3,0}+\frac{1}{r^2}\kappa\otimes T_1
+\sqrt{\frac{8}{r^2+1}}
 \left(\sigma_1\otimes T_2+\sigma_2\otimes T_3\right).
\label{eq:Sp2E3Wedge}
\end{align}
The first one is reducible, whereas the second is irreducible.
\end{thm}
\begin{proof}
We set
\begin{equation*}
	\widetilde{H}:= *F_A + \frac{1}{6} (\omega_I^2 + \omega_J^2 + \omega_K^2) \wedge F_A.
\end{equation*}
Then, \( A\) is an \( \Sp(2)\)-instanton if and only if \( \widetilde{H} = 0\). 
Write
\[
H_i:=(\widetilde{H},T_i),
\qquad i=1,2,3,
\]
so that
\[
\widetilde{H}=H_1\otimes T_1+H_2\otimes T_2+H_3\otimes T_3.
\]
We first consider the $T_1$-component, and wedge the form $H_1$ with the forms \(e^{ij}\). 
In particular, we see
\begin{align}
0=H_1\wedge e^{67} +3 H_1\wedge e^{24}
&=
\frac{16}{3ab}
\left(p_1+\widetilde w_1w_2-w_1\widetilde w_2 \right)\vol_X,
\label{eq:Sp2E3H1a}\\
0=-H_1\wedge e^{57}+3H_1\wedge e^{23}
&=
\frac{16}{3ab}
\left( p_2+w_1w_2+\widetilde w_1\widetilde w_2
\right)\vol_X,
\label{eq:Sp2E3H1b}
\end{align}
Next, adding and subtracting suitable equations obtained by wedging the
$T_2$-component gives
\begin{align}
0=H_2\wedge e^{35}+H_2\wedge e^{46}
&=
-\frac{4}{3bc}
\left(p_1\widetilde w_2+p_2w_2\right)\vol_X,
\label{eq:Sp2E3Kill1}\\
0=-H_2\wedge e^{36}+H_2\wedge e^{45}
&=
-\frac{4}{3bc}
\left(p_1w_2-p_2\widetilde w_2\right)\vol_X.
\label{eq:Sp2E3Kill2}
\end{align}
We note that the corresponding equations obtained from $H_3=0$ give the same
two conditions.
Consequently,
\[
\begin{pmatrix}
\widetilde w_2&w_2\\
w_2&-\widetilde w_2
\end{pmatrix}
\begin{pmatrix}
p_1\\ p_2
\end{pmatrix}
=0.
\]
If $w_2^2+\widetilde w_2^2\neq0$,  the matrix is invertible, and hence $p_1=p_2=0$. If
$w_2=\widetilde w_2=0$, the same conclusion follows from
\eqref{eq:Sp2E3H1a}, \eqref{eq:Sp2E3H1b}. Thus
\begin{equation}
p_1=p_2=0.
\label{eq:Sp2E3pzero}
\end{equation}
Equations \eqref{eq:Sp2E3H1a}, \eqref{eq:Sp2E3H1b} then reduce to
\begin{equation}
w_1w_2+\widetilde w_1\widetilde w_2=0,
\qquad
\widetilde w_1w_2-w_1\widetilde w_2=0.
\label{eq:Sp2E3orth}
\end{equation}
After imposing \eqref{eq:Sp2E3pzero} and \eqref{eq:Sp2E3orth}, the remaining independent equations of the $\Sp(2)$-instanton condition are
\begin{align}
c^2p_0'+2fp_0&=0,
\label{eq:Sp2E3p0}\\
bc\,w_1'+aw_1&=0,
&
bc\,\widetilde w_1'+a\widetilde w_1&=0,
\label{eq:Sp2E3w1}\\
fw_2'+(p_0+1)w_2&=0,
&
f\widetilde w_2'+(p_0+1)\widetilde w_2&=0,
\label{eq:Sp2E3w2}\\
\left(af+bc(p_0-1)\right)w_1&=0,
&
\left(af+bc(p_0-1)\right)\widetilde w_1&=0,
\label{eq:Sp2E3w1alg}\\
\left(ac(p_0+1)-bf\right)w_2&=0,
&
\left(ac(p_0+1)-bf\right)\widetilde w_2&=0,
\label{eq:Sp2E3w2alg}
\end{align}
and
\begin{equation}
c^2p_0
+a^2\left(w_2^2+\widetilde w_2^2\right)
-b^2\left(w_1^2+\widetilde w_1^2\right)+3=0.
\label{eq:Sp2E3moment}
\end{equation}
Equations \eqref{eq:Sp2E3p0} and \eqref{eq:Sp2E3moment}
come from $H_1=0$, while
\eqref{eq:Sp2E3w1}--\eqref{eq:Sp2E3w2alg}
come from $H_2=H_3=0$.
Using \eqref{eq:Metric2} and switching to the $r$-coordinate via
\eqref{eq:dpdt}, the differential equations \eqref{eq:Sp2E3p0}--\eqref{eq:Sp2E3w2} become
\begin{align}
\frac{dp_0}{dr}
&=-\frac{2}{r}p_0,
\label{eq:Sp2E3rp0}\\
\frac{dw_1}{dr}
&=-\frac{r}{r^2+1}w_1,
&
\frac{d\widetilde w_1}{dr}
&=-\frac{r}{r^2+1}\widetilde w_1,
\label{eq:Sp2E3rw1}\\
\frac{dw_2}{dr}
&=-\frac{r^3(p_0+1)}{r^4-1}w_2,
&
\frac{d\widetilde w_2}{dr}
&=-\frac{r^3(p_0+1)}{r^4-1}\widetilde w_2.
\label{eq:Sp2E3rw2}
\end{align}
Moreover, the algebraic equations \eqref{eq:Sp2E3w1alg}--\eqref{eq:Sp2E3moment} reduce to
\begin{equation}
(r^2p_0-1)w_1
=(r^2p_0-1)\widetilde w_1
=(r^2p_0-1)w_2
=(r^2p_0-1)\widetilde w_2
=0,
\label{eq:Sp2E3ralg}
\end{equation}
\begin{equation}
2r^2p_0
+(r^2-1)\left(w_2^2+\widetilde w_2^2\right)
-(r^2+1)\left(w_1^2+\widetilde w_1^2\right)
+6=0.
\label{eq:Sp2E3rmoment}
\end{equation}

Solving equation \eqref{eq:Sp2E3rp0} gives
\[
p_0=\frac{C}{r^2}
\]
for some $C\in\mathbb R$. If $C\neq1$, then
\eqref{eq:Sp2E3ralg} gives
\[
w_1=\widetilde w_1=w_2=\widetilde w_2=0.
\]
Then, equation \eqref{eq:Sp2E3rmoment} gives $C=-3$, yielding the first of our instantons
\[
A_{\mathrm{red}}
=
A_{3,0}-\frac{3}{r^2}\kappa\otimes T_1.
\]
Suppose now that $C=1$. Equations
\eqref{eq:Sp2E3rw1}--\eqref{eq:Sp2E3rw2} give
\begin{align*}
w_1=\frac{C_1}{\sqrt{r^2+1}},
\quad
\widetilde w_1=\frac{\widetilde C_1}{\sqrt{r^2+1}},
\quad
w_2=\frac{C_2}{\sqrt{r^2-1}},
\quad
\widetilde w_2=\frac{\widetilde C_2}{\sqrt{r^2-1}},
\end{align*}
for constants $C_1,\widetilde C_1,C_2,\widetilde C_2\in\mathbb R$, and equations \eqref{eq:Sp2E3orth} and \eqref{eq:Sp2E3rmoment} become
\begin{align*}
C_1C_2+\widetilde C_1\widetilde C_2&=0,\\
\widetilde C_1C_2-C_1\widetilde C_2&=0,\\
C_1^2+\widetilde C_1^2-C_2^2-\widetilde C_2^2&=8.
\end{align*}
The first two equations imply that either
$(C_1,\widetilde C_1)=(0,0)$ or
$(C_2,\widetilde C_2)=(0,0)$. The third equation rules out the first
possibility, and hence
\[
C_2=\widetilde C_2=0,
\qquad
C_1^2+\widetilde C_1^2=8.
\]
Therefore,  writing $C_1=2\sqrt2\cos\theta, \widetilde C_1=2\sqrt2\sin\theta$, we conclude that the following connections are also invariant $\Sp(2)$-instantons
\begin{equation}\label{Awedge}
A_{\wedge,\theta} = A_{3,0} + \frac{1}{r^2}\kappa \otimes T_1 +\sqrt{\frac{8}{r^2+1}} \left( \cos\theta \sigma_1 -  \sin\theta \sigma_2 \right) \otimes T_2 + \sqrt{\frac{8}{r^2+1}}\left( \sin\theta \sigma_1 + \cos\theta \sigma_2 \right) \otimes T_3,
\end{equation}
where $\theta \in \R$.

Next, we prove that for any $\theta \in \R$, the connection \( A_{\wedge, \theta}\) is in the same gauge orbit as $A_{\wedge} := A_{\wedge, 0}$. 
Since $\exp(\theta T_1)$ commutes with every element of $\lambda_3(\U(1))$, it defines an $\SU(3)$-invariant gauge transformation, represented in our pullback convention by $g_\theta=\exp(\theta T_1)$.
Since \( g_{\theta}\) is constant, \( g_{\theta}\) acts on \( A\) purely by the adjoint action 
$
A \mapsto \Ad(g_{\theta})A.
$
Then, since \( T_1\) commutes with itself, it also commutes with \( \exp(\theta T_1)\) so \( g_{\theta}\) fixes \( T_1\). On the other hand, \( \Ad(g_{\theta})\) rotates the \( \left( T_2,T_3 \right)\) plane by the angle \( \theta\). We conclude that \( g_{\theta} A_{\wedge} = A_{\wedge, \theta}\). \\
\indent Finally, a direct computation shows that 
\begin{equation}\label{Sp2E3curv}
	|F_{A_\mathrm{red}}|^2 = \frac{72\left( r^4 + 1 \right)}{r^8}, \qquad
	|F_{A_{\wedge}}|^2 = \frac{72r^6 + 8r^4 + 8r^2 + 8}{r^{10} + r^8},
\end{equation}
and hence $A_{\mathrm{red}}$ and $A_{\wedge}$ are different.
Thus $A_{\mathrm{red}}$ and $A_{\wedge}$ are the only solutions up to
$\SU(3)$-invariant gauge transformations. 
Moreover, $F_{A_{\wedge}}$ has components spanning all of
$\mathfrak{so}(3)$, so $A_{\wedge}$ is irreducible, whereas
$A_{\mathrm{red}}$ reduces to the instanton on the $S^1$-bundle that we also denoted $E_3$.
\end{proof}

We next show that both \( A_{\mathrm{red}}\) and \( A_{\wedge}\) extend. As in the abelian gauge case, we use the removable singularity theorem. However, we now have two different bundles where the connections can extend, $\widetilde{E}_3$ and $\widetilde{E}_\wedge$, and we want to determine which bundle carries each of the connections. 
To do so, we first prove a lemma in greater generality, for non-compact cohomogeneity one manifolds with one singular orbit, 
using the Setup \ref{setup:cohom-one}. The Lemma provides a
compatibility condition between the limit of the connection at the singular orbit and the singular isotropy homomorphism $\chi: H \rightarrow S$ corresponding to the extension bundle $\widetilde{E}_\chi \rightarrow X$, assuming we already know that the connection extends across the singular orbit.

We add the following to our setup.
Fix an \( \mathrm{Ad}(G)\)-invariant inner-product on \( \mathfrak{g}\), and let 
\begin{equation*}
	\mathfrak{g} = \mathfrak{k} \oplus \mathfrak{m} 
\end{equation*}
be the corresponding orthogonal decomposition. Define \( \mathfrak{q} = \mathfrak{h} \cap \mathfrak{m}\) so that \( \mathfrak{h} = \mathfrak{k} \oplus \mathfrak{q}\). Note that \( \mathfrak{q}\) is the space of ``collapsing directions'' as we go towards \( X_{\mathrm{sing}}\). 

Next, identifying \( E_{\lambda} \cong \left( 0,\infty \right) \times P_{\lambda}\) as in Section  \ref{sec:InvarBundConn}, for each \( t\), define a linear map \( \restr{A_t}{\mathfrak{q}}: \mathfrak{q} \rightarrow \mathfrak{s}\) by 
\begin{equation*}
	\restr{A_t}{\mathfrak{q}}(\xi) := A_{(t,\left[ e,e \right]_{\lambda})}(\xi^{\sharp}), \qquad \xi \in \mathfrak{q}.
\end{equation*}
This is the restriction to \( \mathfrak{q}\) of the left-invariant form on \( G\) representing \( A\) at radius \( t\).

\begin{lem}
	Suppose that \( A\) is a \( G\)-invariant connection on \( E_{\lambda} \rightarrow X^{*}\) in temporal gauge that extends across \( G / H\) to a connection \( \widetilde{A}\) on \( \widetilde{E}_{\chi} \rightarrow X\). Then 
	\begin{equation*}
		\lim_{t\rightarrow 0}\restr{A_t}{\mathfrak{q}} = \mathrm{Ad}(s) \circ
        \restr{d\chi_e}{\mathfrak{q}},
	\end{equation*}
	for some \( s \in S\). Consequently, 
	\begin{equation*}
		\mathrm{rank}\left( \restr{\lim_{t\rightarrow 0}A_t}{\mathfrak{q}} \right) = \mathrm{rank}\left( \restr{d\chi_{e}}{\mathfrak{q}} \right).
	\end{equation*}
	\label{lem:rank}
\end{lem}

\begin{proof}
	Let \( \phi:\left( 0,\infty \right) \times G/K \rightarrow X^*\) be the identification map from section 4. Define 
	\begin{equation*}
		\gamma(t) = \phi\left( t,eK \right),\qquad t > 0.
	\end{equation*}
	Set \( x_0 := \lim_{t \rightarrow 0}\gamma(t) \in X_{\mathrm{sing}}\). We have \( \Pi(x_0) = eH\). Consider the basepoints \( \left[ e,e \right]_{\lambda} \in \left( P_{\lambda} \right)_{eK}\) and \( \left[ e,e \right]_{\chi} \in \left( Q_{\chi} \right)_{eH}\). Since \( A\) is in temporal gauge, the curve \( t \mapsto \left( t,\left[ e,e \right]_{\lambda} \right) \in E_{\lambda} \cong \left( 0,\infty \right) \times P_{\lambda}\) is a horizontal lift of \( \gamma\) with respect to \( A\). 
    Let \( \Phi\) be the \( G\)-equivariant isomorphism from Remark \ref{rmk:bundleequivariance}.
    Since \( \Phi^*\widetilde{A} = A\), the curve \( t \mapsto \Phi\left( t,\left[ e,e \right]_{\lambda} \right)\) is an \( \widetilde{A}\)-horizontal lift of \( \gamma\). 
    Since \( \gamma\) extends smoothly to \( t = 0\) (it is a normal geodesic) and \( \widetilde{A}\) is smooth on all of \( \widetilde{E}_{\chi}\), the lift $\Phi\left( t,\left[ e,e \right]_{\lambda} \right)$ extends continuously to \( t = 0\). Set 
    $$
    \tilde{p}_{0} : = \lim_{t \rightarrow 0} \Phi\left( t,\left[ e,e \right]_{\lambda} \right) \in ( \widetilde{E}_{\chi} )_{x_0},
    $$
    and let \( s \in S\) be defined by \( \tilde{p}_{0} = \left[ e,e \right]_{\chi} s^{-1}\). 
	Now, let \( \xi \in \mathfrak{q} \subset \mathfrak{h}\). Write \( \xi^{\sharp}\) for the fundamental vector field of \( \xi\), i.e.\ \( \xi^{\sharp}_{p} = \frac{d}{du} \exp\left( u\xi \right) \cdot p |_{u = 0}\).
    Since \( \Phi\) is \( G\)-equivariant, we have \( d\Phi(\xi^\sharp) = \xi^\sharp \circ \Phi\), so 
	\begin{eqnarray*}
		\restr{A_t}{\mathfrak{q}}(\xi) 
        =
        \widetilde{A}_{\Phi(t,\left[ e,e \right]_{\lambda})}\left( d\Phi\left( \xi^{\sharp}_{(t,\left[ e,e \right]_{\lambda})} \right) \right) 
        =
        \widetilde{A}_{\Phi\left( t,\left[ e,e \right]_{\lambda} \right)}\left( \xi^{\sharp}_{\Phi\left( t,\left[ e,e \right]_{\lambda} \right)} \right)
        \xrightarrow[t \to 0]{} \widetilde{A}_{\tilde{p}_{0}}\left( \xi^{\sharp}_{\tilde{p}_0} \right).
	\end{eqnarray*}
	It remains to compute the limit. First, we see that since \( \xi \in \mathfrak{h}\), the subgroup \( \exp(u\xi) \subset H\) fixes \( x_0 \) and since \( h \cdot \left[ e,e \right]_{\chi} = \left[ e,e \right]_{\chi} \chi(h)\) for all \( h \in H\),
    we have
	\begin{equation*}
		\exp(u\xi)\cdot \tilde{p}_0 = \exp(u\xi)\cdot \left[ e,e \right]_{\chi} s^{-1} = \left[ e,e \right]_{\chi} \chi\left( \exp (u\xi) \right)s^{-1} = \tilde{p}_0 \cdot \left( s \chi\left( \exp (u\xi) \right) s^{-1} \right).
	\end{equation*}
    Moreover, 
	\( \chi(\exp (u\xi)) = \exp\left( u d\chi(\xi) \right)\) since \( \chi\) is a Lie group homomorphism.
    Hence, the following fundamental vector field is vertical:
    $$
    \xi^{\sharp}_{\tilde{p}_0} = \restr{\frac{d}{du} \tilde{p}_0 \cdot \exp\left( u \mathrm{Ad}(s)d\chi(\xi) \right)}{u = 0}.
    $$
    Indeed, 
    it is the fundamental vector field of \( \mathrm{Ad}(s) d\chi(\xi) \in \mathfrak{s}\) for the right \( S\)-action. Since a connection form sends the fundamental vector field of \( \zeta \in \mathfrak{s}\) to \( \zeta\), we conclude that
	\begin{equation*}
		\tilde{A}_{\tilde{p}_{0}}\left( \xi^{\sharp}_{\tilde{p}_0} \right) = \mathrm{Ad}(s) d\chi(\xi).
	\end{equation*}
\end{proof}

\begin{thm}[Extendibility of invariant $\Sp(2)$-instantons on $E_3$]
\label{prop:Sp2E3Extension}
The two $\SU(3)$-invariant $\Sp(2)$-instantons of
Theorem \ref{thm:Sp2E3} extend smoothly across the zero section.
More precisely,
\begin{enumerate}[label=(\alph*)]
\item The reducible instanton $A_{\mathrm{red}}$ extends to
$\widetilde E_3 \to X$.
\item The irreducible instanton $A_{\wedge}$ extends to
$\widetilde E_{\wedge} \to X$.
\end{enumerate}
Consequently, up to gauge transformations, $\widetilde E_3$ admits
a unique $\SU(3)$-invariant $\Sp(2)$-instanton, which is reducible,
whereas $\widetilde E_{\wedge}$ admits a unique
$\SU(3)$-invariant $\Sp(2)$-instanton, which is irreducible.
\end{thm}

\begin{proof}
A direct calculation using the curvature formulas \eqref{Sp2E3curv} shows that the curvatures of both $A_{\mathrm{red}}$ and $A_{\wedge}$ remain bounded on $[ 1, \infty )$. Then, Proposition \ref{prop:ExtendConnection} implies that both connections extend smoothly across the zero section, up to gauge.

It remains to determine which of the two extensions of $E_3$, which we denoted $\widetilde{E}_{3} = \Pi^*Q_3$ and $\widetilde E_{\wedge}=\Pi^*Q_{\wedge}$, carries each connection. 
We recall that $\mathfrak{q}=(X_4, X_1, X_5)$, and compute the ranks from the Lemma above, noting that we can similarly take the limit when $r$ tends to 1.

First, we see that
\begin{equation*}
	\lim_{r\rightarrow 1}A_\mathrm{red} = 3\zeta\otimes T_1-3\kappa\otimes T_1, \qquad
    \lim_{r\rightarrow 1}A_\wedge = 3\zeta\otimes T_1 +\kappa\otimes T_1
    +2\sigma_1\otimes T_2 +2\sigma_2\otimes T_3.
\end{equation*}
Taking the restriction to $\mathfrak{q}$, we see that 
\[
\restr{\lim_{r\rightarrow 1}A_\mathrm{red}}{\mathfrak{q}}: 
X_4\longmapsto -3T_1,\qquad
X_1\longmapsto0,\qquad
X_5\longmapsto0,
\]
while 
\[
\restr{\lim_{r\rightarrow 1}A_\wedge}{\mathfrak{q}}: 
X_4\longmapsto T_1,\qquad
X_1\longmapsto2T_2,\qquad
X_5\longmapsto2T_3,
\]
so their respective ranks are 1 and 3.

We now compute the rank of $d\chi_3 |_e$. Let $Y \in \mathfrak{u}(2)$, we have
\begin{equation*}
	\chi_{3}\left( \mathrm{exp}(uY) \right) 
    = \diag \left( 1, ( \mathrm{exp}(u\, \mathrm{tr}(Y)))^3 \right)
    = \exp\left( -3iu\, \mathrm{tr}(Y) T_1 \right),
\end{equation*}
so differentiating at $u=0$ gives
$
d\chi_3|_e(Y)
=
-3i\operatorname{tr}(Y)\,T_1.
$
To compute the traces of $X_4,X_1,X_5$ as elements of $\mathfrak{u}(2)$ we
identify $\mathfrak{u}(2)\subset\mathfrak{su}(3)$ with its upper-left
$(2\times2)$ block as per \eqref{eq:U2}, and we obtain $i,0,0$, respectively.
Consequently,
\[
\restr{(d\chi_3|_e)}{\mathfrak{q}}: 
X_4\longmapsto 3T_1,\qquad
X_1\longmapsto0,\qquad
X_5\longmapsto0,
\]
so its rank is 1.

To compute the rank of $d\hat{\chi} |_e$, we similarly take $Y \in \mathfrak{u}(2)$ and observe that
\begin{equation*}
\hat{\chi} \left( \mathrm{exp}(uY) \right) 
=\hat{\chi} \left( [\mathrm{exp}(uY_0), z] \right) 
=\varpi (\mathrm{exp}(uY_0)),
\end{equation*}
where $Y_0$ denotes the traceless part of $Y$. Hence, $
d\widehat\chi|_e(Y)
=
d\varpi|_e\left(
Y_0
\right).
$
Noting that $d\varpi|_e (Y_0)(\cdot) =[Y_0, \cdot]$, we conclude that
\[
\restr{(d\hat{\chi}|_e)}{\mathfrak{q}}: 
X_4\longmapsto T_1,\qquad
X_1\longmapsto 2 T_2,\qquad
X_5\longmapsto 2 T_3,
\]
so its rank is 3.

Finally, matching the ranks shows that $A_{\mathrm{red}}$ extends to
$\widetilde E_3 \to X$ and $A_{\wedge}$ extends to $\widetilde E_{\wedge} \to X$.
\end{proof}

\bibliographystyle{plain}
\bibliography{Instanton-Ref}

\Addresses

\end{document}